\documentclass[a4paper,12pt,reqno]{amsart}
\usepackage{amsmath,amssymb,amsthm}
\usepackage{latexsym}
\usepackage{color}
\usepackage{graphicx}
\usepackage{mathrsfs}
\usepackage{enumerate}
\usepackage[abbrev]{amsrefs}
\usepackage[T1]{fontenc}
\usepackage{mathtools}
\mathtoolsset{showonlyrefs=true}

\allowdisplaybreaks[4]

\theoremstyle{plain}
\newtheorem{thm}{Theorem}[section]
\newtheorem{lemm}[thm]{Lemma}
\newtheorem{prop}[thm]{Proposition}

\theoremstyle{definition}

\newtheorem{rem}[thm]{Remark}

\makeatletter

\@addtoreset{equation}{section}
\makeatother

\DeclareFontFamily{U}{matha}{\hyphenchar\font45}
\DeclareFontShape{U}{matha}{m}{n}{
      <5> <6> <7> <8> <9> <10> gen * matha
      <10.95> matha10 <12> <14.4> <17.28> <20.74> <24.88> matha12
      }{}
\DeclareSymbolFont{matha}{U}{matha}{m}{n}
\DeclareFontSubstitution{U}{matha}{m}{n}

\DeclareFontFamily{U}{mathx}{\hyphenchar\font45}
\DeclareFontShape{U}{mathx}{m}{n}{
      <5> <6> <7> <8> <9> <10>
      <10.95> <12> <14.4> <17.28> <20.74> <24.88>
      mathx10
      }{}
\DeclareSymbolFont{mathx}{U}{mathx}{m}{n}
\DeclareFontSubstitution{U}{mathx}{m}{n}

\DeclareMathDelimiter{\vvvert}{0}{matha}{"7E}{mathx}{"17}

\renewcommand{\div}{\operatorname{div}}
\newcommand{\dB}{\dot{B}}

\newcommand{\supp}{\operatorname{supp}}

\renewcommand{\leq}{\leqslant}
\renewcommand{\geq}{\geqslant}

\newcommand{\up}{u_{\rm per}}
\newcommand{\vp}{v_{\rm per}}

\newcommand{\n}[1]{{\left\|#1\right\|}}
\newcommand{\N}[1]{\left\vvvert #1 \right\vvvert}

\newcommand{\lp}[1]{\left[#1\right]}
\renewcommand{\mp}[1]{\left\{#1\right\}}
\renewcommand{\sp}[1]{\left(#1\right)}

\begin{document}
\title[Time-periodic Navier--Stokes flows in $\mathbb{R}^n$ with $n \geq 2$]
{Time-periodic solutions to \\ the Navier--Stokes equations on the whole space \\ including the two-dimensional case}
\author[M.~Fujii]{Mikihiro Fujii}
\address[M.~Fujii]{Institute of Mathematics for Industry, Kyushu University, Fukuoka 819--0395, Japan}
\email[M.~Fujii]{fujii.mikihiro.096@m.kyushu-u.ac.jp}
\begin{abstract}
Let us consider the incompressible Navier--Stokes equations with the time-periodic external forces in the whole space $\mathbb{R}^n$ with $n\geq 2$ and investigate the existence and non-existence of time-periodic solutions. 
In the higher dimensional case $n \geq 3$, we construct a unique small solution for given small time-periodic force in the scaling critical spaces of Besov type and prove its stability under small perturbations.
In contrast, for the two-dimensional case $n=2$, the time-periodic solvability of the Navier--Stokes equations has been long standing open.
It is the central work of this paper that we have now succeeded in solving this issue negatively by providing examples of small external forces such that each of them does not generate time-periodic solutions.
\end{abstract}
\maketitle


\section{Introduction}\label{sec:intro}
We consider the incompressible Navier--Stokes equations with time-periodic external forces on the whole space $\mathbb{R}^n$ with $n \geq 2$:
\begin{align}\label{eq:NS_P}
    \begin{cases}
        \partial_t u -\Delta u + (u \cdot \nabla)u + \nabla p = f, \qquad & t \in \mathbb{R},x\in \mathbb{R}^n,\\
        \div u = 0, \qquad & t \in \mathbb{R},x\in \mathbb{R}^n,
    \end{cases}
\end{align}
where $u=u(t,x):\mathbb{R} \times \mathbb{R}^n \to \mathbb{R}^n$ and $p=p(t,x):\mathbb{R} \times \mathbb{R}^n \to \mathbb{R}$ represent the unknown velocity field and pressure of the fluid, respectively, whereas the given external force $f=f(t,x): \mathbb{R} \times \mathbb{R}^n \to \mathbb{R}^n$ is assumed to be $T$-periodic, that is $f(t+T)=f(t)$ holds for all $t \in \mathbb{R}$ {with some positive constant $T$}.
It is well-known that \eqref{eq:NS_P} possesses the scaling invariant structure, that is, if  $u$ and $p$ solve \eqref{eq:NS_P} with some external force $f$, then
\begin{align}\label{scaling-up}
    u_{\lambda}(t,x):=  \lambda  u (\lambda^2 t,\lambda x),\qquad
    p_{\lambda}(t,x):= \lambda^2 p (\lambda^2 t,\lambda x)
\end{align}
also satisfy \eqref{eq:NS_P} with $f$ replaced by
\begin{align}\label{scaling-f}
    f_{\lambda}(t,x):= \lambda^3 f (\lambda^2 t,\lambda x)
\end{align}
for all $\lambda >0$.
Function spaces of which the norms are invariant {under} the scaling transforms \eqref{scaling-up}-\eqref{scaling-f} are called the scaling critical spaces for \eqref{eq:NS_P}.
The purpose of this paper is to consider the solvability of the time-periodic problem \eqref{eq:NS_P} in the scaling critical Besov-type spaces framework.
In the higher dimensional case $n \geq 3$, we prove the unique existence and global in time stability of the $T$-periodic solutions to \eqref{eq:NS_P}. 
For the two-dimensional case $n=2$, it has been well-known as an open problem whether the time-periodic solution for the two-dimensional incompressible Navier--Stokes equations \eqref{eq:NS_P} with $n=2$ exists or not.
The major outcome in this paper is to solve this question negatively and construct some arbitrarily small external forces, each of which does not produce time-periodic solutions.


\subsection{Known results and the position of our study}
We recall known results for the time-periodic problem of the Navier--Stokes equations on unbounded domains.
It was Maremonti \cites{Mar-91-Non,Mar-91-RM} who first constructed a unique time-periodic solution to \eqref{eq:NS_P} on the three-dimensional whole space $\mathbb{R}^3$ and half space $\mathbb{R}^3_+$.
Kozono and Nakao \cite{Koz-Nak-96} introduced the notion of integral equation \eqref{eq:IE_P} below corresponding to \eqref{eq:NS_P} and showed the existence of a unique small time-periodic mild solution to \eqref{eq:NS_P} in the Lebesgue spaces framework on the whole space $\mathbb{R}^n$, the half space $\mathbb{R}^n_+$ with $n \geq 3$, and the exterior domains in $\mathbb{R}^n$ with $n\geq 4$.
Taniuchi \cite{Tan-99} showed the stability of the global solution to the initial value problem of Navier--Stokes equations with the time-periodic external forces around the time-periodic flow constructed in \cite{Koz-Nak-96} in the framework of weak-mild solutions.
Yamazaki \cite{Yam-00-FE} generalized the results in \cites{Koz-Nak-96} to the Morrey spaces frameworks on $\mathbb{R}^n$ with $n \geq 3$.
Yamazaki \cite{Yam-00-MA} considered the Navier--Stokes equations \eqref{eq:NS_P} with external forces that may not decay as $t \to \infty$, which is a similar situation to the time-periodic setting and proved the global existence of solutions in a scaling critical space $u \in BC(\mathbb{R};L^{n,\infty}(\Omega))$ for given small external force $f$ in the scaling critical class $(-\Delta)^{-\frac{1}{2}}f \in BC(\mathbb{R};L^{\frac{n}{2},\infty}(\Omega))$, where $\Omega$ is the whole space $\mathbb{R}^n$, the half space $\mathbb{R}^n_+$, or the exterior domain in $\mathbb{R}^n$ with $n \geq 3$.
Geissert, Hieber, and Nguyen \cite{Gei-Hie-Ngu-16} proposed a new approach on the time-periodic problem in a general framework and applied it to several viscous incompressible fluids on $\mathbb{R}^n$ with $n \geq 3$ and constructed small time-periodic solutions in a scaling critical space $BC(\mathbb{R};L^{n,\infty}(\mathbb{R}^n))$ if the given time-periodic external force $f=\div F$ with the scaling critical class $F \in BC(\mathbb{R}; L^{\frac{n}{2},\infty}(\mathbb{R}^n))$ is sufficiently small.

For the two-dimensional case, the situation is completely different from that for the higher dimensional case, and there are only few results on the existence of time-periodic solutions in the two-dimensional unbounded domains.
As mentioned in \cite{Gal-13}, the solvability of time-periodic problems in two-dimensional unbounded domains has been known to be as difficult as that of stationary problems.
Indeed, the proofs of all results for higher dimensional case mentioned above completely fails in two-dimensional case. 
One of the reasons for this difficulty is that the decay rate of the heat kernels on $\mathbb{R}^2$ is so slow that it is difficult to find a function space $X$ that establishes the key bilinear estimate 
\begin{align}\label{est:imp2}
    \n{
    \int_{-\infty}^t e^{(t-\tau)\Delta} \mathbb{P}\div (u(\tau) \otimes v(\tau))d\tau
    }_{X}
    \leq{}
    C
    \n{u}_{X}
    \n{v}_{X},
\end{align}
although it is known for the higher dimensional case, such as $BC(\mathbb{R};L^{n,\infty}(\mathbb{R}^n))$ with $n \geq 3$ by \cites{Yam-00-MA,Gei-Hie-Ngu-16}.
{On the other hand, there are some results solving the two-dimensional time periodic problem in the whole plane by considering some special situations.}
In \cite{Gal-13}, Galdi constructed time-periodic solutions to \eqref{eq:NS_P} around the constant flow $\widehat{e}_1=(1,0)^{\rm t}$.
Tsuda \cite{Tsu-18} proved the existence of time-periodic solutions to the compressible Navier--Stokes equations with the given small time-periodic external forces satisfying some spatial antisymmetric conditions. 
However, there is no previous research on time-periodic solvability in two-dimensional unbounded domains without special assumptions such as around non-zero constant equilibrium states or spatial antisymmetry. In particular, the two-dimensional analysis corresponding to the results \cites{Mar-91-Non,Mar-91-RM,Koz-Nak-96,Gei-Hie-Ngu-16,Yam-00-MA,Yam-00-FE} {for} the higher dimensional case mentioned above is completely unresolved.

In this paper, 
we address the solvability of the time-periodic problem \eqref{eq:NS_P} 
not only 
in the higher dimensional case $\mathbb{R}^n$ with $n \geq 3$, 
but also 
in the two-dimensional case $\mathbb{R}^2$, 
and aim to reveal the existence or non-existence of time-periodic solutions in the framework of scaling critical function spaces of Besov type.
More precisely, for the higher dimensional case $\mathbb{R}^n$ with $n \geq 3$,  
we prove that for $1 \leq p < n$ and $1 \leq \sigma \leq \infty$, there exists a unique small time-periodic solution $\up \in \widetilde{C}(\mathbb{R};\dB_{p,\sigma}^{\frac{n}{p}-1}(\mathbb{R}^n))$, provided that the given time-periodic external force $f \in \widetilde{C}(\mathbb{R};\dB_{p,\sigma}^{\frac{n}{p}-3}(\mathbb{R}^n))$ is sufficiently small,
where $\widetilde{C}(\mathbb{R};\dB_{p,\sigma}^{s}(\mathbb{R}^n)):= {C}(\mathbb{R};\dB_{p,\sigma}^{s}(\mathbb{R}^n)) \cap \widetilde{L^{\infty}}(\mathbb{R};\dB_{p,\sigma}^{s}(\mathbb{R}^n))$
and $\widetilde{L^r}(\mathbb{R};\dB_{p,\sigma}^{s}(\mathbb{R}^n))$ is the Chemin--Lerner space; 
the definition and basic properties of this function space are mentioned in Section \ref{sec:pre}.
See Remark \ref{rem:thm1} for the reason why we need the Chemin--Lerner spaces.
For the stability of the time-periodic solution $\up$, 
we consider the initial value problem of the incompressible Navier--Stokes equations with the time-periodic external forces:
\begin{align}\label{eq:NS_IVP_v}
    \begin{cases}
        \partial_t v -\Delta v + (v \cdot \nabla)v + \nabla q = f, \qquad & t >0,x\in \mathbb{R}^n,\\
        \div v = 0, \qquad & t  \geq 0,x\in \mathbb{R}^n,\\
        v(0,x) = v_0 (x), \qquad & x \in \mathbb{R}^n
    \end{cases}
\end{align}
and prove that if the initial disturbance $v_0(x) - \up(0,x)$ is sufficiently small in $\dB_{q,\sigma}^{\frac{n}{q}-1}(\mathbb{R}^n)$ with $1 \leq q < 2n$, then \eqref{eq:NS_IVP_v} possesses a unique mild solution {\it in the strong sense}
\footnote{In the known result \cite{Tan-99}, the time periodic stability is proved in the framework of the mild solutions {\it in the weak sense}; see Remark \ref{rem:thm1-1} below.}
and it holds
\begin{align}
    \lim_{t \to \infty}
    \n{v(t) - \up(t)}_{\dB_{q,\sigma}^{\frac{n}{q}-1}(\mathbb{R}^n)}
    =
    0.
\end{align}
Furthermore, we consider the two-dimensional case and show that the above result on the existence of the time-periodic solution {\it fails}, that is, for each $1\leq p \leq 2$ and $0 < \delta \ll 1$, there exists a time-periodic external force $f_{\delta} \in \widetilde{C}(\mathbb{R};\dB_{p,1}^{\frac{2}{p}-3}(\mathbb{R}^2))$ with the norm less than $\delta$ such that there exists no time-periodic solution to \eqref{eq:NS_P} with the force $f_{\delta}$ in some subset of ${C}(\mathbb{R};\dB_{2,1}^0(\mathbb{R}^2))$.



\subsection{Main results}
Now, we provide the precise statements of our main theorems.
To this end, we recall the notion of mild solutions to \eqref{eq:NS_P} which was proposed by \cite{Koz-Nak-96}.
By the Duhamel principle, the equation \eqref{eq:NS_P} is formally equivalent to 
\begin{align}\label{eq:IE_P}
    u(t) = \int_{-\infty}^t e^{(t-\tau)\Delta} \mathbb{P}f(\tau) d\tau - \int_{-\infty}^t e^{(t-\tau)\Delta} \mathbb{P}\div (u(\tau) \otimes u(\tau)) d\tau,
\end{align}
where 
$\mp{ e^{t\Delta}}_{t>0}$ denotes the heat semigroup, 
and
$\mathbb{P} := I +\nabla \div (-\Delta)^{-1}$ is the Helmholtz projection on $\mathbb{R}^n$.
We say that $u$ is a mild solution to \eqref{eq:NS_P} if 
$u$ satisfies \eqref{eq:IE_P} for all $t \in \mathbb{R}$.
See Section \ref{sec:pre} for the definitions of function spaces appearing in the following theorems.
\subsubsection{Higher dimensional case}
We first focus on the existence and stability of the time-periodic strong solutions to \eqref{eq:NS_P} in higher dimensional whole space case $\mathbb{R}^n$ with $n \geq 3$.
The first main result of this paper now reads:
\begin{thm}[Existence of time-periodic solutions on $\mathbb{R}^n$ with $n \geq 3$]\label{thm:1}
    Let $n \geq 3$ be an integer and let $1 \leq p < n$ and $1 \leq \sigma \leq \infty$.
    Then, there exists a positive constant $\delta_0=\delta_0(n,p,\sigma)$ and $\varepsilon_0=\varepsilon_0(n,p,\sigma)$ such that 
    for any $T>0$ and $T$-periodic external force $ f \in \widetilde{C}(\mathbb{R};\dB_{p,\sigma}^{\frac{n}{p}-3}(\mathbb{R}^n))$ satisfying  
    \begin{align}
        \| f \|_{\widetilde{L^{\infty}}(\mathbb{R};\dB_{p,\sigma}^{\frac{n}{p}-3}(\mathbb{R}^n))} \leq \delta_0,
    \end{align}
    the equation \eqref{eq:NS_P} possesses a unique $T$-periodic mild solution $\up \in U_{p,\sigma}(\mathbb{R}^n)$, where  
    \begin{align}
        U_{p,\sigma}(\mathbb{R}^n)
        :=
        \mp{
        u \in \widetilde{C}(\mathbb{R};\dB_{p,\sigma}^{\frac{n}{p}-1}(\mathbb{R}^n))
        \ ; \ 
        \| u \|_{\widetilde{L^{\infty}}(\mathbb{R};\dB_{p,\sigma}^{\frac{n}{p}-1}(\mathbb{R}^n))} \leq \varepsilon_0
        }.
    \end{align}
    Moreover, there exists a positive constant $C=C(n,p,q)$ 
    such that 
    the solution $\up$ satisfies the following a priori estimate:
    \begin{align}\label{apriori}
        \n{\up}_{\widetilde{L^{\infty}}(\mathbb{R};\dB_{p,\sigma}^{\frac{n}{p}-1}(\mathbb{R}^n))} 
        \leq
        C
        \| f \|_{\widetilde{L^{\infty}}(\mathbb{R};\dB_{p,\sigma}^{\frac{n}{p}-3}(\mathbb{R}^n))}.
    \end{align}
\end{thm}
\begin{rem}\label{rem:thm1}
    We give some remarks on Theorem \ref{thm:1}. 
    \begin{enumerate}
        \item 
        It is generally acknowledged that time-periodic and stationary problems are closely related, and there is a result on the existence of the stationary solutions in Besov spaces framework that corresponds to Theorem \ref{thm:1}.
        In \cite{Kan-Koz-Shi-19}, Kaneko, Kozono, and Shimizu showed that there exists a unique small solution to  the stationary Navier--Stokes equations on the whole space $\mathbb{R}^n$ with $n \geq 3$ in the scaling critical Besov spaces $\dB_{p,\sigma}^{\frac{n}{p}-1}(\mathbb{R}^n)$ for small external forces in $\dB_{p,\sigma}^{\frac{n}{p}-3}(\mathbb{R}^n)$ for $1 \leq p < n$ and $1 \leq \sigma \leq \infty$.
        \item 
        Let us explain why we use not usual space-time Besov spaces $BC(\mathbb{R};\dB_{p,\sigma}^{s}(\mathbb{R}^n))$ but the Chemin--Lerner spaces $\widetilde{C}(\mathbb{R};\dB_{p,\sigma}^{s}(\mathbb{R}^n))$.
        Considering the bilinear estimates \eqref{est:imp2} with $X=BC(\mathbb{R};\dB_{p,\sigma}^{\frac{n}{p}-1}(\mathbb{R}^n))$,
        it is difficult to show it unless $\sigma = \infty$.
        In contrast, if we switch the order of the $L^{\infty}_t$-norm and the $\ell^{\sigma}$-norm, then we {may make use of} the maximal regularity estimate Lemma \ref{lemm:max-reg} below {to} obtain \eqref{est:imp2} with $\widetilde{C}(\mathbb{R};\dB_{p,\sigma}^{\frac{n}{p}-1}(\mathbb{R}^n))$ for all $1 \leq \sigma \leq \infty$; see Remark \ref{rem:max} below for the detail.
        In particular, choosing $\sigma < \infty$ is significant in the time-periodic stability limit \eqref{lim} in Theorem \ref{thm:1-1} below.
    \end{enumerate}
\end{rem}
{
As is mentioned in the above remark, one of the advantage of Chemin--Lerner space is to treat the asymptotic stability of the time periodic solutions.}
For the stability of \eqref{eq:NS_IVP_v} around the time-periodic solution $\up$ constructed in Theorem \ref{thm:1} above, 
we set $w:=v -\up$ and consider the following equations which $w$ should solve:
\begin{align}\label{eq:perturbed}
    \begin{cases}
    \begin{aligned}
    \partial_t w  
    &
    - \Delta w 
    + (w \cdot \nabla)w\\
    &
    + (\up \cdot \nabla) w 
    + (w \cdot \nabla) \up 
    + \nabla \pi = 0,
    \end{aligned}
    & t>0,x \in \mathbb{R}^n,\\
    \div w = 0,
    & t\geq 0,x \in \mathbb{R}^n,\\
    w(0,x) = w_0(x):= v_0(x) - \up(0,x),
    &x \in \mathbb{R}^n.
    \end{cases}
\end{align}
We say that $w$ is a mild solution to \eqref{eq:perturbed} if it solves the following corresponding integral equation:
\begin{align}\label{IE:w}
    \begin{split}
    w(t)
    =
    e^{t\Delta}w_0
    &
    -
    \int_{0}^t
    e^{(t-\tau)\Delta} \mathbb{P}\div (\up(\tau) \otimes w(\tau) + w(\tau) \otimes \up(\tau)) d\tau\\
    &
    -
    \int_0^{t}
    e^{(t-\tau)\Delta} \mathbb{P}\div (w(\tau) \otimes w(\tau)) d\tau.
    \end{split}
\end{align}
Our result on the time-periodic stability reads as follows:
\begin{thm}[Time-periodic stability on $\mathbb{R}^n$ with $n \geq 3$]\label{thm:1-1}
Let $n \geq 3$ be an integer and let $T>0$.
Let $p$, $q$, $r$, and $\sigma$ satisfy
\begin{gather}
    1 \leq p < n,\qquad
    1 \leq q < 2n, \qquad
    \frac{1}{q}
    -\frac{1}{p}
    <
    \frac{1}{n},
    \qquad
    1 \leq \sigma < \infty,\\
    \max
    \mp{
    0,
    1
    -
    \frac{n}{2}
    \min
    \mp{
    1,
    \frac{2}{q},
    {\frac{1}{p}+\frac{1}{q}}
    }
    }
    <
    \frac{1}{r}
    <
    \frac{1}{2}
    -
    \frac{n}{2}
    \max
    \mp{
    0,{\frac{1}{q}-\frac{1}{p}}
    }.
\end{gather}
Then, there exists a positive constants $\delta_0=\delta_0(n,p,q,r,\sigma)$
such that 
if $T$-periodic solution 
$\up \in \widetilde{C}(\mathbb{R};\dB_{p,\sigma}^{\frac{n}{p}-1}(\mathbb{R}^n))$ 
constructed in Theorem \ref{thm:1}
and 
the initial disturbance $w_0 \in \dB_{q,\sigma}^{\frac{n}{q}-1}(\mathbb{R}^n)$
satisfy 
\begin{align}
    \n{\up}_{\widetilde{L^{\infty}}(\mathbb{R};\dB_{p,\sigma}^{\frac{n}{p}-1}(\mathbb{R}^n))}
    \leq
    \delta_0,\qquad
    \n{w_0}_{\dB_{q,\sigma}^{\frac{n}{q}-1}(\mathbb{R}^n)}
    \leq
    \delta_0,
\end{align}
then
\eqref{eq:perturbed} possesses a unique mild solution $w$ in the class
\begin{align}
    w \in 
    \widetilde{C}([0,\infty);\dB_{q,\sigma}^{\frac{n}{q}-1}(\mathbb{R}^n))
    \cap 
    \widetilde{L^r}(0,\infty;\dB_{q,\sigma}^{\frac{n}{q}-1+\frac{2}{r}}(\mathbb{R}^n)).
\end{align}
Moreover, the following asymptotic limit holds 
\begin{align}\label{lim}
    \lim_{t \to \infty}
    \n{w(t)}_{\dB_{q,\sigma}^{\frac{n}{q}-1}(\mathbb{R}^n)}
    =0.
\end{align}
\end{thm}
\begin{rem}\label{rem:thm1-1}
We provide some comments on Theorem \ref{thm:1-1}.
\begin{enumerate}
    \item 
    Theorem \ref{thm:1-1} can be compared with the results of Taniuchi \cite{Tan-99}.
    In his result, the solution to the perturbed equations \eqref{eq:perturbed} should be considered in the framework that $w$ satisfies the integral equation \eqref{IE:w} in the {\it weak sense}.
    This is because 
    if we consider the integral equation \eqref{IE:w} in the strong sense by following the argument in \cite{Tan-99}, 
    we meet a difficulty when controlling the convection terms 
    $ 
    (\up \cdot \nabla) w 
    + 
    (w \cdot \nabla) \up
    $
    in some time-weighted norms like $\sup_{t > 0 }t^{\frac{1}{2}}\n{w(t)}_{L^n(\mathbb{R}^n)}$ since $\up$ does not have any decay structure in time.
    In contrast, our Theorem \ref{thm:1-1} is able to find a solution to \eqref{IE:w} in the {\it strong sense} thanks to the maximal regularity of the heat kernel and bilinear estimates in Chemin--Lerner spaces; see Lemmas \ref{lemm:max-reg} and \ref{lemm:nonlin-ndim} below.
    \item 
    In Theorem \ref{thm:1-1}, the condition $n \geq 3$ is used only for the guarantee of the existence of the time-periodic solution $\up$.
    Therefore, if we obtained a two-dimensional time-periodic solution to \eqref{eq:NS_P} with some external force, then we might obtain the stability result Theorem \ref{thm:1-1} with $n=2$.
    However, as is claimed in Theorem \ref{thm:2} below, the time-periodic problem is not solvable in the two-dimensional case.
\end{enumerate}
\end{rem}
\subsubsection{Two-dimensional case}
Now, we introduce the central work of this paper.
In the following theorem, we claim that Theorem \ref{thm:1} {\it fails} in the two-dimensional case.
\begin{thm}
[Non-existence of the time-periodic solution on $\mathbb{R}^2$]\label{thm:2}
    Let $n=2$ and $1 \leq p \leq 2$.
    Then, there exists a positive constant $\varepsilon_0=\varepsilon_0(p)$ such that for any $0< \delta \leq \varepsilon_0$ and $0 < T \leq 2^{\frac{1}{\delta^2}}$, there exists a $T$-periodic external force $f_{\delta} \in \widetilde{C}(\mathbb{R};\dB_{p,1}^{\frac{2}{p}-3}(\mathbb{R}^2))$ such that 
    \begin{align}
        \| f_{\delta} \|_{\widetilde{L^{\infty}}(\mathbb{R};\dB_{p,1}^{\frac{2}{p}-3}(\mathbb{R}^2))}
        \leq
        \delta
    \end{align}
    and \eqref{eq:NS_P} possesses no $T$-periodic solution belonging to the class $V(\mathbb{R}^2)$, where  
    \begin{align}\label{non-exists-class}
        V(\mathbb{R}^2)
        :=
        \mp{
        u \in BC(\mathbb{R};\dB_{2,1}^0(\mathbb{R}^2))
        \ ; \ 
        \| u(t_0) \|_{\dB_{2,1}^0(\mathbb{R}^2)}
        \leq 
        \varepsilon_0\ {\rm for\ some\ }t_0 \in \mathbb{R}
        }.
    \end{align}
\end{thm}
\begin{rem}\label{rem:thm2}
    We make mention of some remarks.
    \begin{enumerate}
    \item 
    Our non-existence class $V(\mathbb{R}^2)$ may include functions with arbitrarily large $L^{\infty}(\mathbb{R};\dB_{2,1}^0(\mathbb{R}^2))$-norms.
    From this and 
    $\widetilde{C}(\mathbb{R};\dB_{p,1}^{\frac{2}{p}-1}(\mathbb{R}^2)) 
    \hookrightarrow 
    BC(\mathbb{R};\dB_{2,1}^0(\mathbb{R}^2))$ for $1 \leq p \leq 2$,
    we see that 
    $U_{p,1}(\mathbb{R}^2) \subsetneq V(\mathbb{R}^2)$ with $1 \leq p \leq 2$,
    which implies 
    Theorem \ref{thm:2} claims a stronger results than the negative proposition of Theorem \ref{thm:1} with $n=2$, $1 \leq p \leq 2$, and $\sigma =1$.
    Moreover, even if each $f_{\delta}$ generates a time-periodic solution $u_{{\rm per},\delta}$ in a wider class than $V(\mathbb{R}^2)$, then $u_{{\rm per},\delta} \notin V(\mathbb{R}^2)$ implies that
    $\n{u_{{\rm per},\delta}}_{\widetilde{L^{\infty}}(\mathbb{R};\dB_{p,1}^{\frac{2}{p}-1}(\mathbb{R}^2))}$ must be bounded from below by a positive constant $\varepsilon_0$ independent of $\delta$ although $\| f_{\delta} \|_{\widetilde{L^{\infty}}(\mathbb{R};\dB_{p,1}^{\frac{2}{p}-3}(\mathbb{R}^2))}$ vanishes as $\delta \downarrow 0$; this means that the a priori estimate \eqref{apriori} never holds.
    \item 
    We compare Theorem \ref{thm:2} with the results in \cite{Tsu-18}, where it was shown that the two-dimensional compressible Navier--Stokes equations with time-periodic external forces on the whole plane possesses the small time-periodic solution if the given time-periodic external force satisfy some spatial antiasymmetric conditions.
    In contrast, our external force has a anisotropic structure due to $\cos(Mx_1)$ with some $M \gg 1$; see \eqref{ex:f} below for the detail.
    Thus, it is crucial to impose a certain spatial symmetry for external forces in order to construct a two-dimensional time-periodic solution.
    
    \end{enumerate}
\end{rem}
Let us explain the idea of the proof of Theorem \ref{thm:2}.
We use the contradiction argument.
For $0< \delta \ll 1$ and $f_{\delta}$ proposed in \eqref{ex:f} below, there exists a $T$-periodic solution $u_{{\rm per},\delta} \in C(\mathbb{R};\dB_{2,1}^0(\mathbb{R}^2))$ with $\n{u_{{\rm per},\delta}(t_0)}_{\dB_{2,1}^0(\mathbb{R}^2)} \leq \varepsilon_0$ for some $t_0 \in \mathbb{R}$ and $\varepsilon_0>0$.
Then, using the method for ill-posedness,
Proposition \ref{prop:non-per} below enables us to construct external forces $f_{\delta}$ for each $0 < \delta \leq \varepsilon_0$ such that
there exists a Navier--Stokes flow $u \in C([t_0,t_0+kT];\dB_{2,1}^0(\mathbb{R}^2))$ started at $t=t_0$ satisfying the initial condition $u(t_0) = u_{{\rm per},\delta}(t_0)$ 
and the estimate 
\begin{align}
    \n{u(t_0 + kT)}_{\dB_{2,1}^0(\mathbb{R}^2)} \geq 2\varepsilon_0
\end{align}
for some large $k \in \mathbb{N}$.
Then, since $\n{u(t_0)}_{\dB_{2,1}^0(\mathbb{R}^2)}  = \n{u_{{\rm per},\delta}(t_0)}_{\dB_{2,1}^0(\mathbb{R}^2)} \leq \varepsilon_0$, we see that $u(t_0) \neq u(t_0+kT)$, which means $u$ is not $T$-periodic.
However, since it follows from Proposition \ref{prop:unique} that the uniqueness in $C([t_0,t_0+kT];\dB_{2,1}^0(\mathbb{R}^2))$ holds for solutions to the two-dimensional Navier--Stokes equations, we see that $u=u_{{\rm per},\delta}$ on $[t_0,t_0+kT]$ and 
\begin{align}
    u_{{\rm per},\delta}(t_0) = u(t_0) \neq u(t_0+kT) = u_{{\rm per},\delta}(t_0+kT),
\end{align}
which contradicts the assumption that $u_{{\rm per},\delta}$ is $T$-periodic and completes the proof.
\subsection{Organization and {notations} on this paper}
This paper is organized as follows.
In Section \ref{sec:pre}, we recall the definitions of Besov and Chemin--Lerner spaces and their basic properties.
We focus on the higher dimensional case in Section \ref{sec:nD} and provide the proofs of Theorems \ref{thm:1} and \ref{thm:1-1}.
In Section \ref{sec:2D}, we prove Propositions \ref{prop:non-per} on the construction of non-time-periodic solutions and \ref{prop:unique} for the unconditional uniqueness of two-dimensional Navier--Stokes flow to complete the proof of Theorem \ref{thm:2}.
In Appendix \ref{sec:a}, we note some remarks on the para differential calculus in Chemin--Lerner spaces.

Throughout this paper, we denote by $C$ and $c$ the constants, which may differ in each line. In particular, $C=C(*,...,*)$ denotes the constant which depends only on the quantities appearing in parentheses. 
For any $T>0$, we say that a function $f=f(t)$ on $\mathbb{R}$ is $T$-periodic if $f(t+T)=f(t)$ holds for all $t \in \mathbb{R}$.

\section{Preliminaries}\label{sec:pre}
In this section, we prepare notations used in this paper and recall the definition of several function spaces and their basic properties which are frequently used in this paper.

We recall the definitions of Besov and Chemin--Lerner spaces.
Let $\mathscr{S}(\mathbb{R}^n)$ be the set of all Schwartz functions on $\mathbb{R}^n$, and 
we denote by {$\mathscr{S}'(\mathbb{R}^n)$} the set of all tempered distributions on $\mathbb{R}^n$.
Let $\varphi_0 \in \mathscr{S}(\mathbb{R}^n)$ satisfy 
\begin{align}
    \supp \widehat{\varphi_0} \subset \mp{ \xi \in \mathbb{R}^n\ ;\ 2^{-1} \leq | \xi | \leq 2 },\quad
    0 \leq \widehat{\varphi_0}(\xi) \leq 1, 
\end{align}
and 
\begin{align}
    \sum_{j \in \mathbb{Z}}
    \widehat{\varphi_j}(\xi) = 1 \qquad {\rm for\ all\ }\xi \in  \mathbb{R}^n \setminus \{0\},
\end{align}
where we have set $\varphi_j(x):=2^{nj}\varphi_0(2^jx)$.
{Using them, we define the dyadic frequency localized operators by
\begin{align}
    \Delta_jf:=\mathscr{F}^{-1}\lp{\widehat{\varphi_j}(\xi)\widehat{f}}, \qquad f \in \mathscr{S}'(\mathbb{R}^n),\ j \in \mathbb{Z}.
\end{align}}
For $1 \leq p,\sigma \leq \infty$ and $s \in \mathbb{R}$, the Besov space $\dB_{p,\sigma}^s(\mathbb{R}^n)$ is defined as 
\begin{align}
    \dB_{p,\sigma}^s(\mathbb{R}^n)
    :={}&
    \mp{
    f \in \mathscr{S}'(\mathbb{R}^n)/\mathscr{P}(\mathbb{R}^n)
    \ ; \ 
    \n{f}_{\dB_{p,\sigma}^s(\mathbb{R}^n)}<\infty
    },\\
    \n{f}_{\dB_{p,\sigma}^s(\mathbb{R}^n)}
    :={}&
    \n{
    \mp{
    2^{sj}
    \n{\Delta_j f}_{L^p(\mathbb{R}^n)}
    }_{j \in \mathbb{Z}}
    }_{\ell^{\sigma}(\mathbb{Z})},
\end{align}
where $\mathscr{P}(\mathbb{R}^n)$ is the set of all polynomials on $\mathbb{R}^n$.
It is well-known that if $s <n/p$ or $(s,\sigma) = (n/p,1)$, then it holds
\begin{align}
    \dB_{p,\sigma}^s(\mathbb{R}^n)
    \sim
    \mp{
    f \in \mathscr{S}'(\mathbb{R}^n)
    \ ; \ 
    \n{f}_{\dB_{p,\sigma}^s(\mathbb{R}^n)}<\infty,\quad
    f = \sum_{j \in \mathbb{Z}}\Delta_j f \quad {\rm in\ }\mathscr{S}'(\mathbb{R}^n)
    }.
\end{align}
For $1 \leq p,r,\sigma \leq \infty$, $s \in \mathbb{R}$, and an interval $I \subset \mathbb{R}$, we define 
the Chemin--Lerner space $\widetilde{L^r}(I;\dB_{p,\sigma}^s(\mathbb{R}^n))$ by 
\begin{align}
    \widetilde{L^r}(I;\dB_{p,\sigma}^s(\mathbb{R}^n))
    :={}&
    \mp{
    F:I \to  \mathscr{S}'(\mathbb{R}^n)/\mathscr{P}(\mathbb{R}^n)
    \ ; \ 
    \n{F}_{\widetilde{L^r}(I;\dB_{p,\sigma}^s(\mathbb{R}^n))}
    <\infty
    },\\
    \n{F}_{\widetilde{L^r}(I;\dB_{p,\sigma}^s(\mathbb{R}^n))}
    :={}&
    \n{
    \mp{
    2^{sj}
    \n{\Delta_j F}_{L^r(I;L^p(\mathbb{R}^n))}
    }_{j \in \mathbb{Z}}
    }_{\ell^{\sigma}(\mathbb{Z})}.
\end{align}
Since $\dot{H}^s(\mathbb{R}^n)=\dB_{2,2}^s(\mathbb{R}^n)$,
we write
\begin{align}
    \widetilde{L^r}(I;\dot{H}^s(\mathbb{R}^n))
    :={}&
    \widetilde{L^r}(I;\dB_{2,2}^s(\mathbb{R}^n)).
\end{align}
We also use the following notation
\begin{align}
    \widetilde{C}(I ; \dB_{p,\sigma}^s(\mathbb{R}^n))
    :=
    C(I ; \dB_{p,\sigma}^s(\mathbb{R}^n))
    \cap
    \widetilde{L^{\infty}}(I ; \dB_{p,\sigma}^s(\mathbb{R}^n)).
\end{align}
The Chemin--Lerner spaces were first introduced by \cite{Che-Ler-95} and continue to be frequently used for the analysis of compressible viscous fluids in critical Besov spaces.
The Chemin--Lerner spaces possess similar embedding properties as that for usual Besov spaces:
\begin{itemize}
    \item []
        $\widetilde{L^r}(I; \dB_{p,\sigma_1}^s(\mathbb{R}^n)) 
        \hookrightarrow
        \widetilde{L^r}(I; \dB_{p,\sigma_2}^s(\mathbb{R}^n))$ 
        for $1 \leqslant \sigma_1 \leqslant \sigma_2 \leqslant \infty$,
    \item []
        $\widetilde{L^r}(I; \dB_{p_1,\sigma}^{s+\frac{n}{p_1}}(\mathbb{R}^n)) 
        \hookrightarrow
        \widetilde{L^r}(I; \dB_{p_2,\sigma}^{s+\frac{n}{p_2}}(\mathbb{R}^n))$ 
        for $1 \leqslant p_1 \leqslant p_2 \leqslant \infty$.
\end{itemize}
It also holds by the Hausdorff--Young inequality that
\begin{align}
    &
    \widetilde{L^r}(I; \dB_{p,\sigma}^s(\mathbb{R}^n))
    \hookrightarrow
    {L^r}(I; \dB_{p,\sigma}^s(\mathbb{R}^n))
    \ {\rm for\ }1 \leqslant \sigma \leqslant r \leqslant \infty,\\
    &
    {L^r}(I; \dB_{p,\sigma}^s(\mathbb{R}^n)) 
    \hookrightarrow
    \widetilde{L^r}(I; \dB_{p,\sigma}^s(\mathbb{R}^n))
    \ {\rm for\ }1 \leqslant r \leqslant \sigma \leqslant \infty.
\end{align}
See \cite{Bah-Che-Dan-11} for more precise information of the Chemin--Lerner spaces.
One advantage of using the Chemin--Lerner spaces is that there holds the following maximal regularity estimates for the heat kernel $e^{t\Delta}$.
\begin{lemm}\label{lemm:max-reg}
Let $n \in \mathbb{N}$.
Then, there exists a positive constant $C=C(n)$ such that 
\begin{align}
    \n{
    e^{(t-t_0)\Delta}a
    }_{\widetilde{L^r}(I;\dB_{p,\sigma}^{s+\frac{2}{r}}(\mathbb{R}^n))}
    &\leq
    C
    \n{a}_{\dB_{p,\sigma}^s(\mathbb{R}^n)}, \label{est:max-1}
    \\
    \n{\int_{t_0}^t
    e^{(t-\tau)\Delta}f(\tau)d\tau
    }_{\widetilde{L^r}(I;\dB_{p,\sigma}^{s+\frac{2}{r}}(\mathbb{R}^n))}
    &\leq
    C
    \n{f}_{\widetilde{L^{r_1}}(I;\dB_{p,\sigma}^{s-2+\frac{2}{r_1}}(\mathbb{R}^n))} \label{est:max-2}
\end{align}
for all $I=(t_0,t_1) \subset \mathbb{R}$, $1 \leq p,\sigma \leq \infty$, $1 \leq r_1 \leq r \leq \infty$, $s \in \mathbb{R}$, $a \in \dB_{p,\sigma}^s(\mathbb{R}^n)$, and $f \in \widetilde{L^{r_1}}(I;\dB_{p,\sigma}^{s+\frac{2}{r_1}}(\mathbb{R}^n))$.
\end{lemm}
\begin{rem}\label{rem:max}
If we attempt to show \eqref{est:max-1} with the norm of left hand side replaced by that for the usual time-space Besov norm $L^r(I;\dB_{p,\sigma}^{s+\frac{2}{r}}(\mathbb{R}^n))$, 
we fail by the following argument: 
\begin{align}
    \n{
    e^{t\Delta}a
    }_{L^r(I;\dB_{p,\sigma}^{s+\frac{2}{r}}(\mathbb{R}^n))}
    ={}&
    \sp{
    \int_{t_0}^{t_1}
    \n{e^{(t-t_0)\Delta}a}_{\dB_{p,\sigma}^{s+\frac{2}{r}}(\mathbb{R}^n)}^r
    dt
    }^{\frac{1}{r}}\\
    \leq{}&
    C
    \lp{
    \int_{t_0}^{t_1}
    \mp{(t-t_0)^{-\frac{1}{r}}\n{a}_{\dB_{p,\sigma}^s(\mathbb{R}^n)}}^r
    dt
    }^{\frac{1}{r}}
    =\infty
\end{align}
{for non-zero $a$,}
where we have used the smoothing estimate for the heat kernel (see \cite{Koz-Oga-Tan-03}*{Lemma 2.2}).
In contrast, we succeed to obtain the maximal regularity estimate by {change} the order of the $L^r_t$-norm and $\ell^{\sigma}$-norm in the Besov norm.
The similar situation as above holds for \eqref{est:max-2}.
\end{rem}
\begin{proof}[Proof of Lemma \ref{lemm:max-reg}]
Although the proof is immediately obtained by \cite{Bah-Che-Dan-11}*{Corollary 2.5}, we shall give the outline of the proof for the readers' convenience. 
It follows from \cite{Bah-Che-Dan-11}*{Lemma 2.4} that there exists an absolute positive constant $C_*$ such that
\begin{align}\label{pf:max-1-1}
    2^{\frac{2}{r}j}
    \left\| \Delta_j e^{(t-t_0)\Delta}a \right\|_{L^p(\mathbb{R}^n)} 
    &\leqslant 
    C_*
    e^{-C_*^{-1}(t-t_0)2^{2j}}
    \| \Delta_j a \|_{L^p(\mathbb{R}^n)}
\end{align}
for all $j \in \mathbb{Z}$.
Taking $L^r(I)$ norm of \eqref{pf:max-1-1}, we see that
\begin{align}
    \| \Delta_j e^{(t-t_0)\Delta}a \|_{L^r(I;L^p(\mathbb{R}^n))} 
    &\leqslant C_*\left\| e^{-C_*^{-1}(t-t_0)2^{2j}} \right\|_{L^r(t_0,\infty)} \| \Delta_j F \|_{L^p(\mathbb{R}^n)}\\
    &= C_*(C_*^{-1}r)^{-\frac{1}{r}}2^{-\frac{2}{r}j}\| \Delta_j F \|_{L^p(\mathbb{R}^n)}.
\end{align}
As $(C_*^{-1}r)^{-\frac{1}{r}}$ is bounded with respect to $r$, we obtain \eqref{est:max-1} by multiplying \eqref{pf:max-1-1} by $2^{sj}$ and taking $\ell^{\sigma}$-norm{. Thus,} we complete the proof of \eqref{est:max-1}.

We next prove \eqref{est:max-2}. 
Let $1 \leqslant r_2 \leqslant \infty$ satisfy $1+1/r=1/r_2+1/r_1$.
It follows from \eqref{pf:max-1-1} and the Hausdorff-Young inequality for the time convolution that
\begin{align}
    \left\| \Delta_j \int_{t_0}^t e^{(t-\tau)\Delta} f(\tau) d\tau \right\|_{L^{r}(I;L^p(\mathbb{R}^n))}
    &\leqslant
    C_*
    \left\|  \int_{t_0}^t e^{-C_*^{-1}(t-\tau)2^{2j}} \| \Delta_j f(\tau)\|_{L^p(\mathbb{R}^n)} d\tau \right\|_{L^{r}(I)}\\
    &\leqslant
    C_*
    \left\| e^{-C_*^{-1}t2^{2j}} \right\|_{L^{r_2}(0,\infty)}
    \| \Delta_j f \|_{L^{r_1}(I;L^p(\mathbb{R}^n))}\\
    &=
    C_*
    (C_*^{-1}r_2)^{-\frac{1}{r_2}}
    2^{(-2+\frac{2}{r_1}-\frac{2}{r})j}
    \| \Delta_j f \|_{L^r(I;L^p(\mathbb{R}^n))}
\end{align}
Note that $(C_*^{-1}r_2)^{-\frac{1}{r_2}}$ is bounded with respect to $1 \leqslant r_2 \leqslant \infty$.
Thus, we obtain \eqref{est:max-2} by multiplying \eqref{pf:max-1-1} by $2^{(s+\frac{2}{r})j}$ and taking $\ell^{\sigma}$-norm, we complete the proof of \eqref{est:max-1}.
\end{proof}
Combining Lemma \ref{lemm:max-reg} and the para-differential calculus in Appendix \ref{sec:a}, we obtain the following bilinear estimate.
\begin{lemm}\label{lemm:nonlin-ndim}
Let $n \geq 2$ be an integer.
Let $1 \leq p,q,\sigma \leq \infty$ and ${2} \leq r \leq r_1\leq \infty$ satisfy
\begin{gather}
    1 \leq p,q,r,\sigma \leq \infty,\qquad
    r \leq r_1 \leq \infty,\\
    \max
    \mp{
    0,n\sp{\frac{1}{q}-\frac{1}{p}}
    }
    < 1 - \frac{2}{r},\qquad
    \min\mp{n,n\sp{\frac{1}{p}+\frac{1}{q}}}-2+\frac{2}{r}>0.
    \label{assump:pqr}
\end{gather}
Then, there exists a positive constant $C=C(n,p,q,r,\sigma)$ such that
\begin{align}
    &\n{
    \int_{t_0}^t
    e^{(t-\tau)\Delta}
    \mathbb{P}\div(u(\tau) \otimes v(\tau))d\tau
    }_{\widetilde{L^{r_1}}(I;\dB_{q,\sigma}^{\frac{n}{q}-1+\frac{2}{r_1}}(\mathbb{R}^n))}\\
    &\quad\leq
    C
    \n{u}_{\widetilde{L^{\infty}}(I;\dB_{p,\sigma}^{\frac{n}{p}-1}(\mathbb{R}^n))}
    \n{v}_{\widetilde{L^{r}}(I;\dB_{q,\sigma}^{\frac{n}{q}-1+\frac{2}{r}}(\mathbb{R}^n))}
\end{align}
for all 
$I=(t_0,t_1) \subset \mathbb{R}$, 
$u \in \widetilde{L^{\infty}}(I;\dB_{p,\sigma}^{\frac{n}{p}-1}(\mathbb{R}^n))$, 
and 
$v \in \widetilde{L^r}(I;\dB_{q,\sigma}^{\frac{n}{q}-1+\frac{2}{r}}(\mathbb{R}^n))$.
\end{lemm}
\begin{rem}\label{rem:nonlin}
In the proof of Theorem \ref{thm:1}, we use the case $p=q$ and $r=\infty$:
\begin{align}
    &\n{\int_{t_0}^t
    e^{(t-\tau)\Delta}
    \mathbb{P}{\div}(u(\tau) \otimes v(\tau))d\tau
    }_{\widetilde{L^{\infty}}(I;\dB_{p,\sigma}^{\frac{n}{p}-1}(\mathbb{R}^n))}\\
    &\quad\leq
    C_1
    \n{u}_{\widetilde{L^{\infty}}(I;\dB_{p,\sigma}^{\frac{n}{p}-1}(\mathbb{R}^n))}
    \n{v}_{\widetilde{L^{\infty}}(I;\dB_{p,\sigma}^{\frac{n}{p}-1}(\mathbb{R}^n))},
\end{align}
where we need to assume $n \geq 3$ and $1 \leq p < n$ due to the conditions \eqref{assump:pqr}.
\end{rem}
\begin{proof}[Proof of Lemma \ref{lemm:nonlin-ndim}]
It follows from Lemma \ref{lemm:max-reg} that
\begin{align}
    &
    \n{
    \int_{t_0}^t
    e^{(t-\tau)\Delta}
    \mathbb{P}\div(u(\tau) \otimes v(\tau))d\tau
    }_{\widetilde{L^r}(I;\dB_{q,\sigma}^{\frac{n}{q}-1+\frac{2}{r}}(\mathbb{R}^n))}\\
    &\quad 
    \leq
    C
    \n{u \otimes v}_{\widetilde{L^r}(I;\dB_{q,\sigma}^{\frac{n}{q}-2+\frac{2}{r}}(\mathbb{R}^n))}\\
    &\quad 
    \leq
    C
    \sum_{k,\ell=1}^n
    \n{T_{u_k}v_{\ell}}_{\widetilde{L^r}(I;\dB_{q,\sigma}^{\frac{n}{q}-2+\frac{2}{r}}(\mathbb{R}^n))}
    +
    C
    \sum_{k,\ell=1}^n
    \n{T_{v_{\ell}} u_k}_{\widetilde{L^r}(I;\dB_{q,\sigma}^{\frac{n}{q}-2+\frac{2}{r}}(\mathbb{R}^n))}\\
    &\qquad
    +
    C
    \sum_{k,\ell=1}^n
    \n{R(u_k,v_{\ell})}_{\widetilde{L^r}(I;\dB_{q,\sigma}^{\frac{n}{q}-2+\frac{2}{r}}(\mathbb{R}^n))}\\
    &\quad
    =:
    A_1[u,v]+A_2[u,v]+A_3[u,v].
\end{align}
{
See Appendix \ref{sec:a} for the definitions of the Bony decompositions $T_fg$ and $R(f,g)$.
}
First, we consider the estimate of $A_1[u,v]$.
By Lemma \ref{lemm:para} (1), we have
\begin{align}
    A_1[u,v]
    \leq{}&
    C
    \n{u}_{\widetilde{L^{\infty}}(I;\dB_{\infty,\sigma}^{-1}(\mathbb{R}^n))}
    \n{v}_{\widetilde{L^{r}}(I;\dB_{q,\sigma}^{\frac{n}{q}-1+\frac{2}{r}}(\mathbb{R}^n))}\\
    \leq{}&
    C
    \n{u}_{\widetilde{L^{\infty}}(I;\dB_{p,\sigma}^{\frac{n}{p}-1}(\mathbb{R}^n))}
    \n{v}_{\widetilde{L^{r}}(I;\dB_{q,\sigma}^{\frac{n}{q}-1+\frac{2}{r}}(\mathbb{R}^n))}.
\end{align}
Next, we focus on the estimate of $A_2[u,v]$.
For the case of $p \leq q$,
it holds by Lemma \ref{lemm:para} and $-1+2/r<0$ that 
\begin{align}
    A_2[u,v]
    \leq{}&
    C
    \n{v}_{\widetilde{L^{r}}(I;\dB_{\infty,\sigma}^{-1+\frac{2}{r}}(\mathbb{R}^n))}
    \n{u}_{\widetilde{L^{\infty}}(I;\dB_{q,\sigma}^{\frac{n}{q}-1}(\mathbb{R}^n))}
    \\
    \leq{}&
    C
    \n{v}_{\widetilde{L^r}(I;\dB_{q,\sigma}^{\frac{n}{q}-1+\frac{2}{r}}(\mathbb{R}^n))}
    \n{u}_{\widetilde{L^{\infty}}(I;\dB_{p,\sigma}^{\frac{n}{p}-1}(\mathbb{R}^n))}
    {.}
\end{align}
For the case of $q \leq p$,
we define $1 \leq \theta \leq \infty$ by $1/q = 1/\theta + 1/p$. 
Then, using $q \leq \theta$, $1/q = 1/\theta + 1/p$, and ${n}/{\theta}-1+{2}/{r} = n(1/q - 1/p) - 1 + 2/r < 0$, we see that 
\begin{align}
    A_2[u,v]
    \leq{}&
    C
    \n{v}_{\widetilde{L^{r}}(I;\dB_{\theta,\sigma}^{\frac{n}{\theta}-1+\frac{2}{r}}(\mathbb{R}^n))}
    \n{u}_{\widetilde{L^{\infty}}(I;\dB_{p,\sigma}^{\frac{n}{p}-1}(\mathbb{R}^n))}
    \\
    \leq{}&
    C
    \n{v}_{\widetilde{L^r}(I;\dB_{q,\sigma}^{\frac{n}{q}-1+\frac{2}{r}}(\mathbb{R}^n))}
    \n{u}_{\widetilde{L^{\infty}}(I;\dB_{p,\sigma}^{\frac{n}{p}-1}(\mathbb{R}^n))}.
\end{align}
Finally, we consider the estimate of $A_3[u,v]$.
For the case of $1/p + 1/q \geq 1$, it holds by Lemma \ref{lemm:para} (2) and $n-2 +2/r >0$ that
\begin{align}
    A_3[u,v]
    \leq{}&
    C
    \sum_{k,\ell=1}^n
    \n{R(u_k,v_{\ell})}_{\widetilde{L^r}(I;\dB_{1,\sigma}^{n-2+\frac{2}{r}}(\mathbb{R}^n))} \\
    \leq{}&
    C
    \n{u}_{\widetilde{L^{\infty}}(I;\dB_{q',\sigma}^{\frac{n}{q'}-1}(\mathbb{R}^n))}
    \n{v}_{\widetilde{L^{r}}(I;\dB_{q,\sigma}^{\frac{n}{q}-1+\frac{2}{r}}(\mathbb{R}^n))}\\
    \leq{}&
    C
    \n{u}_{\widetilde{L^{\infty}}(I;\dB_{p,\sigma}^{\frac{n}{p}-1}(\mathbb{R}^n))}
    \n{v}_{\widetilde{L^{r}}(I;\dB_{q,\sigma}^{\frac{n}{q}-1+\frac{2}{r}}(\mathbb{R}^n))},
\end{align}
where we have used $p \leq q'$.
For the case of $1/p + 1/q \leq 1$, we define $1 \leq \zeta \leq \infty$ by $1/\zeta = 1/p + 1/q$.
Then, we have by {$n/\zeta -2 + 2/r = n(1/p + 1/q) -2 + 2/r >0$} that
\begin{align}
    A_3[u,v]
    \leq{}&
    C
    \sum_{k,\ell=1}^n
    \n{R(u_k,v_{\ell})}_{\widetilde{L^r}(I;\dB_{\zeta,\sigma}^{\frac{n}{\zeta}-2+\frac{2}{r}}(\mathbb{R}^n))} \\
    \leq{}&
    C
    \n{u}_{\widetilde{L^{\infty}}(I;\dB_{p,\sigma}^{\frac{n}{p}-1}(\mathbb{R}^n))}
    \n{v}_{\widetilde{L^{r}}(I;\dB_{q,\sigma}^{\frac{n}{q}-1+\frac{2}{r}}(\mathbb{R}^n))}
    {.}
\end{align}
Hence, we complete the proof.
\end{proof}
\section{Higher dimensional analysis: Proofs of Theorems \ref{thm:1} and \ref{thm:1-1}}\label{sec:nD}
In this section, we provide the proofs of our main theorems on the higher dimensional case.
We are ready to prove Theorems \ref{thm:1} and \ref{thm:1-1}.
\begin{proof}[Proof of Theorem \ref{thm:1}]
By Lemma \ref{lemm:max-reg} and Lemma \ref{lemm:nonlin-ndim} with $p=q$ and $r=\infty$ (see also Remark \ref{rem:nonlin}),
there exists a positive constant $C_1=C_1(n,p,\sigma)$ such that 
\begin{align}
    &\n{\int_{-\infty}^t
    e^{(t-\tau)\Delta}
    \mathbb{P}f(\tau)d\tau
    }_{\widetilde{L^{\infty}}(\mathbb{R};\dB_{p,\sigma}^{\frac{n}{p}-1}(\mathbb{R}^n))}
    \leq 
    C_1 
    \n{f}_{\widetilde{L^{\infty}}(\mathbb{R};\dB_{p,\sigma}^{\frac{n}{p}-3}(\mathbb{R}^n))},
    \\
    &
    \begin{aligned}
    &\n{\int_{-\infty}^t
    e^{(t-\tau)\Delta}
    \mathbb{P}{\div}(u(\tau) \otimes v(\tau))d\tau
    }_{\widetilde{L^{\infty}}(\mathbb{R};\dB_{p,\sigma}^{\frac{n}{p}-1}(\mathbb{R}^n))}\\
    &\quad\leq
    C_1
    \n{u}_{\widetilde{L^{\infty}}(\mathbb{R};\dB_{p,\sigma}^{\frac{n}{p}-1}(\mathbb{R}^n))}
    \n{v}_{\widetilde{L^{\infty}}(\mathbb{R};\dB_{p,\sigma}^{\frac{n}{p}-1}(\mathbb{R}^n))}
    \end{aligned}
\end{align}
for all $f \in \widetilde{L^{\infty}}(\mathbb{R};\dB_{p,\sigma}^{\frac{n}{p}-3}(\mathbb{R}^n))$ and $u,v \in \widetilde{L^{\infty}}(\mathbb{R};\dB_{p,\sigma}^{\frac{n}{p}-1}(\mathbb{R}^n))$.
Now, let $f$ be a $T$-periodic external force satisfying
\begin{align}
    f \in \widetilde{C}(\mathbb{R};\dB_{p,\sigma}^{\frac{n}{p}-3}(\mathbb{R}^n)),\qquad
    \n{f}_{\widetilde{L^{\infty}}(\mathbb{R};\dB_{p,\sigma}^{\frac{n}{p}-3}(\mathbb{R}^n))}
    \leq
    \frac{1}{16C_1^2}.
\end{align}
We consider the map
\begin{align}
    \Phi[u](t):=
    \int_{-\infty}^t
    e^{(t-\tau)\Delta}
    \mathbb{P}f(\tau)d\tau
    -
    \int_{-\infty}^t
    e^{(t-\tau)\Delta}
    \mathbb{P}{\div}(u(\tau)\otimes u(\tau))d\tau
\end{align}
on the complete metric space {$(S_{p,\sigma},d_{S_{p,\sigma}})$, which is defined by}
\begin{align}
    &
    S_{p,\sigma}
    :=
    \mp{
    u \in \widetilde{C}(\mathbb{R};\dB_{p,\sigma}^{\frac{n}{p}-1}(\mathbb{R}^n))\ ;\ 
    \begin{aligned}
    &
    u(t+T)=u(t) \quad {\rm for\ all\ }t \in \mathbb{R},\\
    &
    \n{u}_{\widetilde{L^{\infty}}(\mathbb{R};\dB_{p,\sigma}^{\frac{n}{p}-1}(\mathbb{R}^n))}
    \leq 
    2C_1 
    \n{f}_{\widetilde{L^{\infty}}(\mathbb{R};\dB_{p,\sigma}^{\frac{n}{p}-3}(\mathbb{R}^n))}
    \end{aligned}
    },\\
    &
    {d_{S_{p,\sigma}}(u,\widetilde{u})
    :=
    \n{u-\widetilde{u}}_{\widetilde{L^{\infty}}(0,\infty;\dB_{p,\sigma}^{\frac{n}{p}-1}(\mathbb{R}^n))}.}
\end{align}
Then, for any $u \in S_{p,\sigma}$, since $\Phi[u]$ is $T$-periodic and satisfies
\begin{align}
    \n{\Phi[u]}_{\widetilde{L^{\infty}}(\mathbb{R};\dB_{p,\sigma}^{\frac{n}{p}-1}(\mathbb{R}^n))}
    \leq{}&
    C_1
    \n{f}_{\widetilde{L^{\infty}}(\mathbb{R};\dB_{p,\sigma}^{\frac{n}{p}-3}(\mathbb{R}^n))}
    +
    C_1
    \n{u}_{\widetilde{L^{\infty}}(\mathbb{R};\dB_{p,\sigma}^{\frac{n}{p}-1}(\mathbb{R}^n))}^2\\
    \leq{}&
    C_1
    \n{f}_{\widetilde{L^{\infty}}(\mathbb{R};\dB_{p,\sigma}^{\frac{n}{p}-3}(\mathbb{R}^n))}
    +
    4C_1^3
    \n{f}_{\widetilde{L^{\infty}}(\mathbb{R};\dB_{p,\sigma}^{\frac{n}{p}-3}(\mathbb{R}^n))}^2\\
    \leq{}&
    2C_1
    \n{f}_{\widetilde{L^{\infty}}(\mathbb{R};\dB_{p,\sigma}^{\frac{n}{p}-3}(\mathbb{R}^n))},
\end{align}
we see that $\Phi[u] \in S_{p,\sigma}$.
For $u,v \in {S_{p,\sigma}}$, there holds by Lemma \ref{lemm:nonlin-ndim} with $p=q$ and $r=\infty$ that
\begin{align}\label{u-v}
    \begin{split}
    \Phi[u](t) - \Phi[v](t)
    ={}&
    -\int_{-\infty}^t
    e^{(t-\tau)\Delta}
    \mathbb{P}\div(u(\tau) \otimes (u(\tau)-v(\tau)))d\tau\\
    &-\int_{-\infty}^t
    e^{(t-\tau)\Delta}
    \mathbb{P}\div((u(\tau)-v(\tau)) \otimes v(\tau))d\tau,
    \end{split}
\end{align}
which and Lemma \ref{lemm:nonlin-ndim} with $p=q$ and $r=\infty$ imply 
\begin{align}
    &
    \n{\Phi[u] - \Phi[v]}_{\widetilde{L^{\infty}}(\mathbb{R};\dB_{p,\sigma}^{\frac{n}{p}-1}(\mathbb{R}^n))}\\
    &\quad
    \leq
    C_1\sp{
    \n{u}_{\widetilde{L^{\infty}}(\mathbb{R};\dB_{p,\sigma}^{\frac{n}{p}-1}(\mathbb{R}^n))}
    +
    \n{v}_{\widetilde{L^{\infty}}(\mathbb{R};\dB_{p,\sigma}^{\frac{n}{p}-1}(\mathbb{R}^n))}
    }
    \n{u - v}_{\widetilde{L^{\infty}}(\mathbb{R};\dB_{p,\sigma}^{\frac{n}{p}-1}(\mathbb{R}^n))}\\
    &\quad 
    \leq 
    4C_1^2
    \n{f}_{\widetilde{L^{\infty}}(\mathbb{R};\dB_{p,\sigma}^{\frac{n}{p}-1}(\mathbb{R}^n))}
    \n{u - v}_{\widetilde{L^{\infty}}(\mathbb{R};\dB_{p,\sigma}^{\frac{n}{p}-1}(\mathbb{R}^n))}\\
    &\quad
    \leq 
    \frac{1}{4}
    \n{u - v}_{\widetilde{L^{\infty}}(\mathbb{R};\dB_{p,\sigma}^{\frac{n}{p}-1}(\mathbb{R}^n))}.
\end{align}
Hence, $\Phi[\cdot]$ is a contraction map on $S_{p,\sigma}$ and the Banach fixed point theorem implies there exists a unique $\up \in S_{p,\sigma}$ such that $\up=\Phi[\up]$, which yields a $T$-periodic mild solution satisfying \eqref{apriori}.

For the uniqueness of $T$-periodic solutions, let us assume a $T$-periodic external force generates two $T$-periodic mild solutions $\up$ and $\vp$ to \eqref{eq:NS_P} in the following class:
\begin{align}
    \mp{
    u \in \widetilde{C}(\mathbb{R};\dB_{p,\sigma}^{\frac{n}{p}-1}(\mathbb{R}^n))
    \ ; \ 
    \| u \|_{\widetilde{L^{\infty}}(\mathbb{R};\dB_{p,\sigma}^{\frac{n}{p}-1}(\mathbb{R}^n))} \leq \frac{1}{4C_1}
    }.
\end{align}
Then, using the similar observation as in \eqref{u-v},
we have
\begin{align}
    &
    \n{\up - \vp}_{\widetilde{L^{\infty}}(\mathbb{R};\dB_{p,\sigma}^{\frac{n}{p}-1}(\mathbb{R}^n))}\\
    &\quad
    \leq
    C_1
    \n{\up}_{\widetilde{L^{\infty}}(\mathbb{R};\dB_{p,\sigma}^{\frac{n}{p}-1}(\mathbb{R}^n))}
    \n{\up - \vp}_{\widetilde{L^{\infty}}(\mathbb{R};\dB_{p,\sigma}^{\frac{n}{p}-1}(\mathbb{R}^n))}\\
    &\qquad
    +
    C_1
    \n{\vp}_{\widetilde{L^{\infty}}(\mathbb{R};\dB_{p,\sigma}^{\frac{n}{p}-1}(\mathbb{R}^n))}
    \n{\up - \vp}_{\widetilde{L^{\infty}}(\mathbb{R};\dB_{p,\sigma}^{\frac{n}{p}-1}(\mathbb{R}^n))}\\
    &\quad 
    \leq 
    \frac{1}{2}
    \n{\up - \vp}_{\widetilde{L^{\infty}}(\mathbb{R};\dB_{p,\sigma}^{\frac{n}{p}-1}(\mathbb{R}^n))},
\end{align}
which implies $\up = \vp$.
Thus, we complete the proof.
\end{proof}

\begin{proof}[Proof of Theorem \ref{thm:1-1}]
From Lemmas \ref{lemm:max-reg} and \ref{lemm:nonlin-ndim} that there exists a positive constant $C_1=C_1(n,p,q,r,\sigma)$ such that 
\begin{align}
    &
    \n{e^{t\Delta}w_0}_{\widetilde{L^{\infty}}(0,\infty;\dB_{q,\sigma}^{\frac{n}{q}-1}(\mathbb{R}^n)) \cap \widetilde{L^r}(0,\infty;\dB_{q,\sigma}^{\frac{n}{q}-1+\frac{2}{r}}(\mathbb{R}^n))}
    \leq{}
    C_1
    \n{w_0}_{\dB_{q,\sigma}^{\frac{n}{q}-1}(\mathbb{R}^n)},\\
    &
    \n{
    \int_{t_0}^t
    e^{(t-\tau)\Delta}
    \mathbb{P}\div(u(\tau) \otimes v(\tau))d\tau
    }_{\widetilde{L^{\infty}}(0,\infty;\dB_{q,\sigma}^{\frac{n}{q}-1}(\mathbb{R}^n)) \cap \widetilde{L^r}(0,\infty;\dB_{q,\sigma}^{\frac{n}{q}-1+\frac{2}{r}}(\mathbb{R}^n))}\\
    &
    \qquad
    \leq{}
    C_1
    \n{u}_{\widetilde{L^{\infty}}(\mathbb{R};\dB_{p,\sigma}^{\frac{n}{p}-1}(\mathbb{R}^n))}
    \n{v}_{\widetilde{L^r}(0,\infty;\dB_{q,\sigma}^{\frac{n}{q}-1+\frac{2}{r}}(\mathbb{R}^n))},\\
    &\n{
    \int_{t_0}^t
    e^{(t-\tau)\Delta}
    \mathbb{P}\div(v(\tau) \otimes w(\tau))d\tau
    }_{\widetilde{L^{\infty}}(0,\infty;\dB_{q,\sigma}^{\frac{n}{q}-1}(\mathbb{R}^n)) \cap \widetilde{L^r}(0,\infty;\dB_{q,\sigma}^{\frac{n}{q}-1+\frac{2}{r}}(\mathbb{R}^n))}\\
    &\qquad
    \leq{}
    C_2
    \n{v}_{\widetilde{L^{\infty}}(0,\infty;\dB_{q,\sigma}^{\frac{n}{q}-1}(\mathbb{R}^n)) }
    \n{w}_{\widetilde{L^r}(0,\infty;\dB_{q,\sigma}^{\frac{n}{q}-1+\frac{2}{r}}(\mathbb{R}^n))}
\end{align}
for all 
$w_0 \in \dB_{q,\sigma}^{\frac{n}{p}-1}(\mathbb{R}^n)$, 
$u \in \widetilde{L^{\infty}}(I;\dB_{p,\sigma}^{\frac{n}{p}-1}(\mathbb{R}^n))$ 
and 
$v,w \in \widetilde{L^{\infty}}(0,\infty;\dB_{q,\sigma}^{\frac{n}{q}-1}(\mathbb{R}^n)) \cap \widetilde{L^r}(0,\infty;\dB_{q,\sigma}^{\frac{n}{q}-1+\frac{2}{r}}(\mathbb{R}^n))$.
We assume that the time-periodic solution $\up$ and the initial disturbance $w_0 \in \dB_{q,\sigma}^{\frac{n}{q}-1}(\mathbb{R}^n)$ satisfy
\begin{align}
    \n{\up}_{\widetilde{L^{\infty}}(\mathbb{R};\dB_{p,\sigma}^{\frac{n}{p}-1}(\mathbb{R}^n))}
    \leq
    \frac{1}{8C_1^2},\qquad
    \n{w_0}_{\dB_{q,\sigma}^{\frac{n}{q}-1}(\mathbb{R}^n)}
    \leq
    \frac{1}{8C_1^2}
    {.}
\end{align}
To construct a mild solution to \eqref{eq:perturbed}, we consider a map
\begin{align}
    \Psi[w](t)
    :=
    e^{t\Delta}w_0
    &
    -
    \int_{0}^t
    e^{(t-\tau)\Delta} \mathbb{P}\div (\up(\tau) \otimes w(\tau) + w(\tau) \otimes \up(\tau)) d\tau\\
    &
    -
    \int_0^{t}
    e^{(t-\tau)\Delta} \mathbb{P}\div (w(\tau) \otimes w(\tau)) d\tau.
\end{align}
{We define}
the complete metric space $(S_{q,\sigma}^r, d_{S_{q,\sigma}^r})$ {by}
\begin{align}
    &
    S_{q,\sigma}^r
    :=
    \mp{
    \begin{aligned}
    w \in 
    {}&
    \widetilde{C}([0,\infty);\dB_{q,\sigma}^{\frac{n}{q}-1}(\mathbb{R}^n)) 
    \cap 
    \widetilde{L^r}(0,\infty;\dB_{q,\sigma}^{\frac{n}{q}-1+\frac{2}{r}}(\mathbb{R}^n))
    \ ; \\
    &
    \n{w}_{\widetilde{L^{\infty}}(0,\infty;\dB_{q,\sigma}^{\frac{n}{q}-1}(\mathbb{R}^n)) \cap \widetilde{L^r}(0,\infty;\dB_{q,\sigma}^{\frac{n}{q}-1+\frac{2}{r}}(\mathbb{R}^n))}
    \leq 
    2C_1\n{w_0}_{\dB_{q,\sigma}^{\frac{n}{q}-1}(\mathbb{R}^n)}
    \end{aligned}},\\
    &
    d_{S_{q,\sigma}^r}(w,\widetilde{w})
    :=
    \n{w-\widetilde{w}}_{\widetilde{L^{\infty}}(0,\infty;\dB_{q,\sigma}^{\frac{n}{q}-1}(\mathbb{R}^n)) \cap \widetilde{L^r}(0,\infty;\dB_{q,\sigma}^{\frac{n}{q}-1+\frac{2}{r}}(\mathbb{R}^n))}.
\end{align}
Then, for any $w, \widetilde{w} \in S_{q,\sigma}^r$, there holds
\begin{align}
    &
    \n{\Psi[w]}_{\widetilde{L^{\infty}}(0,\infty;\dB_{q,\sigma}^{\frac{n}{q}-1}(\mathbb{R}^n)) \cap \widetilde{L^r}(0,\infty;\dB_{q,\sigma}^{\frac{n}{q}-1+\frac{2}{r}}(\mathbb{R}^n))}\\
    &\quad\leq{}
    C_1\n{w_0}_{\dB_{q,\sigma}^{\frac{n}{q}-1}(\mathbb{R}^n)}\\
    &\qquad
    +
    2C_1
    \n{\up}_{\widetilde{L^{\infty}}(\mathbb{R};\dB_{p,\sigma}^{\frac{n}{p}-1}(\mathbb{R}^n))}
    \n{w}_{\widetilde{L^r}(0,\infty;\dB_{q,\sigma}^{\frac{n}{q}-1+\frac{2}{r}}(\mathbb{R}^n))}\\
    &\qquad
    +
    C_1
    \n{w}_{\widetilde{L^{\infty}}(0,\infty;\dB_{q,\sigma}^{\frac{n}{q}-1}(\mathbb{R}^n)) }
    \n{w}_{\widetilde{L^r}(0,\infty;\dB_{q,\sigma}^{\frac{n}{q}-1+\frac{2}{r}}(\mathbb{R}^n))}\\
    &\quad
    \leq{}
    C_1\n{w_0}_{\dB_{q,\sigma}^{\frac{n}{q}-1}(\mathbb{R}^n)}\\
    &\qquad
    +
    4C_1^3
    \n{\up}_{\widetilde{L^{\infty}}(\mathbb{R};\dB_{p,\sigma}^{\frac{n}{p}-1}(\mathbb{R}^n))}
    \n{w_0}_{\dB_{q,\sigma}^{\frac{n}{q}-1}(\mathbb{R}^n)}
    +
    4
    C_1^3
    \sp{\n{w_0}_{\dB_{q,\sigma}^{\frac{n}{q}-1}(\mathbb{R}^n)}}^2\\
    &\quad\leq{}
    2C_1\n{w_0}_{\dB_{q,\sigma}^{\frac{n}{q}-1}(\mathbb{R}^n)}
\end{align}
and 
\begin{align}
    &\n{\Psi[w]-\Psi[\widetilde{w}]}_{\widetilde{L^{\infty}}(0,\infty;\dB_{q,\sigma}^{\frac{n}{q}-1}(\mathbb{R}^n)) \cap \widetilde{L^r}(0,\infty;\dB_{q,\sigma}^{\frac{n}{q}-1+\frac{2}{r}}(\mathbb{R}^n))}\\
    &\quad 
    \leq{}
    2C_1
    \n{\up}_{\widetilde{L^{\infty}}(\mathbb{R};\dB_{p,\sigma}^{\frac{n}{p}-1}(\mathbb{R}^n))}
    \n{w-\widetilde{w}}_{\widetilde{L^r}(0,\infty;\dB_{q,\sigma}^{\frac{n}{q}-1+\frac{2}{r}}(\mathbb{R}^n))}\\
    &\qquad
    +
    C_1
    \sp{
    \n{w}_{\widetilde{L^{\infty}}(0,\infty;\dB_{q,\sigma}^{\frac{n}{q}-1}(\mathbb{R}^n)) } 
    + 
    \n{\widetilde{w}}_{\widetilde{L^{\infty}}(0,\infty;\dB_{q,\sigma}^{\frac{n}{q}-1}(\mathbb{R}^n))}
    }
    \n{w-\widetilde{w}}_{\widetilde{L^r}(0,\infty;\dB_{q,\sigma}^{\frac{n}{q}-1+\frac{2}{r}}(\mathbb{R}^n))}\\
    &\quad
    \leq{}
    2C_1^2
    \sp{
    \n{\up}_{\widetilde{L^{\infty}}(\mathbb{R};\dB_{p,\sigma}^{\frac{n}{p}-1}(\mathbb{R}^n))}
    +
    \n{w_0}_{\dB_{q,\sigma}^{\frac{n}{q}-1}(\mathbb{R}^n)}}
    \n{w-\widetilde{w}}_{\widetilde{L^r}(0,\infty;\dB_{q,\sigma}^{\frac{n}{q}-1+\frac{2}{r}}(\mathbb{R}^n))}\\
    &\quad 
    \leq{}
    \frac{1}{2}
    \n{w-\widetilde{w}}_{\widetilde{L^{\infty}}(0,\infty;\dB_{q,\sigma}^{\frac{n}{q}-1}(\mathbb{R}^n)) \cap \widetilde{L^r}(0,\infty;\dB_{q,\sigma}^{\frac{n}{q}-1+\frac{2}{r}}(\mathbb{R}^n))},
\end{align}
which implies that $\Psi[\cdot]$ is a contraction map on $(S_{q,\sigma}^r,d_{S_{q,\sigma}^r})$.
Hence, it follows from the Banach fixed point theorem that there exists a unique $w \in S_{q,\sigma}^r$ such that $w =\Psi[w]$, which yields a mild solution to \eqref{eq:perturbed}.
The uniqueness in $\widetilde{C}([0,\infty);\dB_{q,\sigma}^{\frac{n}{q}-1}(\mathbb{R}^n)) \cap \widetilde{L^r}(0,\infty;\dB_{q,\sigma}^{\frac{n}{q}-1+\frac{2}{r}}(\mathbb{R}^n))$ is a straightforward argument.

Finally, we show \eqref{lim}.
Let $T'>T>0$.
It holds 
\begin{align}
    w(t)
    =
    e^{(t-T)\Delta}w(T)
    &
    -
    \int_T^t
    e^{(t-\tau)\Delta} \mathbb{P}\div (\up(\tau) \otimes w(\tau) + w(\tau) \otimes \up(\tau)) d\tau\\
    &
    -
    \int_T^t
    e^{(t-\tau)\Delta} \mathbb{P}\div (w(\tau) \otimes w(\tau)) d\tau
\end{align}
for $t>T$.
Then, similarly as above, we have 
\begin{align}
    \n{w}_{\widetilde{L^{\infty}}(T',\infty;\dB_{q,\sigma}^{\frac{n}{q}-1}(\mathbb{R}^n))}
    \leq{}&
    \n{e^{(t-T)\Delta}w(T)}_{\widetilde{L^{\infty}}(T',\infty;\dB_{q,\sigma}^{\frac{n}{q}-1}(\mathbb{R}^n))}\\
    &
    +
    2C_1
    \n{\up}_{\widetilde{L^{\infty}}(\mathbb{R};\dB_{p,\sigma}^{\frac{n}{p}-1}(\mathbb{R}^n))}
    \n{w}_{\widetilde{L^{r}}(T,\infty;\dB_{q,\sigma}^{\frac{n}{q}-1+\frac{2}{r}}(\mathbb{R}^n))}\\
    &
    +
    C_1
    \n{w}_{\widetilde{L^{\infty}}(0,\infty;\dB_{q,\sigma}^{\frac{n}{q}-1}(\mathbb{R}^n))}
    \n{w}_{\widetilde{L^r}(T,\infty;\dB_{q,\sigma}^{\frac{n}{q}-1+\frac{2}{r}}(\mathbb{R}^n))},
\end{align}
which implies 
{
\begin{align}
    \n{w}_{\widetilde{L^{\infty}}(T',\infty;\dB_{q,\sigma}^{\frac{n}{q}-1}(\mathbb{R}^n))}
    \leq{}&
    C
    \n{e^{(t-T)\Delta}w(T)}_{\widetilde{L^{\infty}}(T',\infty;\dB_{q,\sigma}^{\frac{n}{q}-1}(\mathbb{R}^n))}\\
    &+
    C
    \n{w}_{\widetilde{L^r}(T,\infty;\dB_{q,\sigma}^{\frac{n}{q}-1+\frac{2}{r}}(\mathbb{R}^n))}.
\end{align}
Thus, we have}
\begin{align}
    \n{w(T')}_{\dB_{q,\sigma}^{\frac{n}{q}-1}(\mathbb{R}^n)}
    \leq{}
    &
    \n{w}_{\widetilde{L^{\infty}}(T',\infty;\dB_{q,\sigma}^{\frac{n}{q}-1}(\mathbb{R}^n))}
    \\
    \leq{}&
    C
    \mp{
    \sum_{j \in \mathbb{Z}}
    \sp{
    e^{-c2^{2j}(T'-T)}
    2^{(\frac{n}{q}-1)j}
    \n{\Delta_j w(T)}_{L^q}
    }^{\sigma}
    }^{\frac{1}{\sigma}}
    \\
    &
    +
    C
    \mp{
    \sum_{j \in \mathbb{Z}}
    \sp{
    2^{(\frac{n}{q}-1+\frac{2}{r})j}
    \n{\Delta_j w}_{L^r(T,\infty;L^q)}
    }^{\sigma}
    }^{\frac{1}{\sigma}}.
\end{align}
Hence, letting $T' \to \infty$ and then letting $T \to \infty$, 
we complete the proof.
\end{proof}

\section{Two-dimensional analysis: Proof of Theorem \ref{thm:2}}\label{sec:2D}
The goal of this section is to prove Theorem \ref{thm:2}.
To this end, we investigate some properties of the following initial value problem of the Navier--Stokes equations:
\begin{align}\label{eq:NS_I}
    \begin{cases}
        \partial_t u - \Delta u + (u \cdot \nabla)u + \nabla p = f, \qquad & t > t_0, x \in \mathbb{R}^2,\\
        \div u = 0,\qquad & t \geq t_0, x \in \mathbb{R}^2,\\
        u(t_0,x) = a(x), \qquad & t=t_0, x \in \mathbb{R}^2,
    \end{cases}
\end{align}
where $t_0 \in \mathbb{R}$ is a given initial time and $a=a(x)$ is a given initial data. 
We say that $v$ is a mild solution to \eqref{eq:NS_I} if it satisfies
\begin{align}\label{eq:IE_u}
    u(t) 
    = 
    e^{(t-t_0)\Delta}a
    + 
    \int_{t_0}^t e^{(t-\tau)\Delta}{\mathbb{P}f}(\tau)d\tau
    -
    \int_{t_0}^t e^{(t-\tau)\Delta}\mathbb{P}\div(u(\tau)\otimes u(\tau))d\tau.
\end{align}

\subsection{Construction of non-periodic in time mild solutions}
The aim of this subsection is to show that for any initial data, there exists a mild solution to \eqref{eq:NS_I} that is not $T$-periodic if we choose an appropriate {time periodic} external force.
More precisely, we prove the following proposition.
\begin{prop}\label{prop:non-per}
Let $1 \leq p \leq 2$.
Then, there exists a positive constant $\varepsilon_0=\varepsilon_0(p)$ such that the following statement holds. 
For any $0 < \delta \leq  \varepsilon_0$, and $0<T\leq 2^{\frac{1}{\delta^2}}$, 
there exist a $T$-periodic external force 
$f_{\delta} \in \widetilde{C}(\mathbb{R};\dB_{p,1}^{\frac{2}{p}-3}(\mathbb{R}^2))$ with 
\begin{align}
    \n{f_{\delta}}_{\widetilde{L^{\infty}}(\mathbb{R};\dB_{p,1}^{\frac{2}{p}-3}(\mathbb{R}^2))} \leq \delta
\end{align}
and
a $k_{\delta,T} \in \mathbb{N}$
such that
for
any {initial time}
$t_0 \in \mathbb{R}$
and
initial data
$a \in \dB_{2,1}^0(\mathbb{R}^2)$ with 
\begin{align}\label{a:small}
    \div a = 0,\qquad
    \n{a}_{\dB_{2,1}^0(\mathbb{R}^2)} \leq \varepsilon_0,
\end{align}
\eqref{eq:NS_I} 
possesses
a mild solution $u_{\delta}[a] \in \widetilde{C}([t_0,t_0+k_{\delta,T}T];\dB_{2,1}^0(\mathbb{R}^2))$
satisfying 
\begin{align}\label{est:u}
    \n{u_{\delta}[a](t_0+k_{\delta,T}T)}_{\dB_{2,1}^0(\mathbb{R}^2)} \geq 2\varepsilon_0.
\end{align}
\end{prop}
\begin{rem}
    By \eqref{a:small} and \eqref{est:u} we see that 
    \begin{align}
        u_{\delta}[a](t_0) \neq u_{\delta}[a](t_0 + k_{\delta,T}T),
    \end{align}
    which implies the solution $u_{\delta}[a]$ is never $T$-periodic, 
    regardless of the choice of small initial data $a\in \dB_{2,1}^0(\mathbb{R}^2)$.
\end{rem}
Before starting the proof of Proposition \ref{prop:non-per}, we mention the idea and outline of it.
In order to obtain \eqref{est:u}, we shall follow the method used in the context of ill-posedness \cites{Bou-Pav-08,Yon-10,Wan-15,Fujii-pre-2,Bej-Tao-06} and
{we deduce the lower bound estimate by the second iteration of the standard successive approximation system.
To this end, we} decompose a solution $u$ of \eqref{eq:NS_I} into the first iteration, the second iteration, and the remainder part:
\begin{align}
    u(t) ={} u^{(1)}(t) + u^{(2)}(t) + \widetilde{u}(t),
\end{align}
where $u^{(1)}$ and $u^{(2)}$ solve the first iterative system
\begin{align}
    &
    \begin{cases}
        \partial_t u^{(1)} - \Delta u^{(1)} + \nabla p^{(1)} = f_{\delta}, \qquad & t > t_0,x \in \mathbb{R}^2,\\
        \div u^{(1)} = 0, & t \geq t_0,x \in \mathbb{R}^2,\\
        u^{(1)}(t_0,x) = 0, & x \in \mathbb{R}^2,
    \end{cases}
\end{align}
and the second iterative system
\begin{align}
    \begin{cases}
        \partial_t u^{(2)} - \Delta u^{(2)} + (u^{(1)} \cdot \nabla)u^{(1)} + \nabla p^{(2)} = 0, \qquad & t > t_0,x \in \mathbb{R}^2,\\
        \div u^{(2)} = 0, & t \geq t_0,x \in \mathbb{R}^2,\\
        u^{(2)}(t_0,x) = 0, & x \in \mathbb{R}^2,
    \end{cases}
\end{align}
respectively,
and the remainder $\widetilde{u}$ should be a solution to 
\begin{align}
    \begin{cases}
    \begin{aligned}
    \partial_t \widetilde{u} - \Delta \widetilde{u}
    &
    +(u^{(1)} \cdot \nabla)u^{(2)}
    +(u^{(2)} \cdot \nabla)u^{(1)}
    +(u^{(2)} \cdot \nabla)u^{(2)}\\
    &
    +(u^{(1)} \cdot \nabla)\widetilde{u}
    +(u^{(2)} \cdot \nabla)\widetilde{u}
    +(\widetilde{u} \cdot \nabla)u^{(1)}
    +(\widetilde{u} \cdot \nabla)u^{(2)}\\
    &
    +(\widetilde{u} \cdot \nabla)\widetilde{u}
    +\nabla \widetilde{p} = 0, 
    \end{aligned}
    &t>t_0,x \in \mathbb{R}^2,\\
    \div \widetilde{u}=0, & t \geq t_0,x \in \mathbb{R}^2,\\
    \widetilde{u}(t_0,x) = a(x), & x \in \mathbb{R}^2.
    \end{cases}
\end{align}
Note that we regard the linear term $e^{(t-t_0)\Delta}a$ as not a part of the first iteration $u^{(1)}$ but a piece of the remainder $\widetilde{u}$; this allows us to obtain the lower-bound estimate for the second iteration $u^{(2)}$ with arbitrariness in the choice of the initial data $a$ since $u^{(2)}$ is independent of $a$.
For sufficiently small $0 < \delta \ll 1$,
choosing 
a suitable interval $I_{\delta}=[t_0,t_0+k_{\delta,T}T]$ with some large $k_{\delta,T} \in \mathbb{N}$ 
and 
a suitable external force $f_{\delta}$ with $\n{f_{\delta}}_{\widetilde{L^{\infty}}(\mathbb{R};\dB_{p,1}^{\frac{2}{p}-3}(\mathbb{R}^2))} \leq \delta$, 
we have 
\begin{align}
    &
    \n{u^{(1)}}_{\widetilde{L^{\infty}}(I_{\delta};\dB_{2,1}^{0}(\mathbb{R}^2))} \leq C \delta,
    \qquad
    \n{u^{(2)}(t_0+k_{\delta,T}T)}_{\dB_{2,1}^{0}(\mathbb{R}^2)} \geq cM^2
\end{align}
for some positive constant $M \gg 1$ independent of $\delta$.
Hence, the essential part of the proof is to construct the remainder $\widetilde{u}$.
However, since
the estimate
\begin{align}\label{est:imp}
    \n{
    \int_{t_0}^t e^{(t-\tau)\Delta} \mathbb{P}\div (v(\tau) \otimes w(\tau))d\tau
    }_{X}
    \leq{}
    C
    \n{v}_{X}
    \n{w}_{X}
\end{align}
fails with $X=\widetilde{L^{\infty}}(I_{\delta};\dB_{2,1}^{0}(\mathbb{R}^2))$, 
ingenuity is needed to achieve the objectives.
To overcome this, we define a norm 
\begin{align}
    \N{v}_{\delta,I}
    :=
    \n{v}_{\widetilde{L^{\infty}}(I;\dB_{2,1}^{0}(\mathbb{R}^2))}
    +
    \frac{1}{\delta}
    \n{v}_{\widetilde{L^{\frac{2}{\delta^2}}}(I;\dot{H}^{\delta^2}(\mathbb{R}^2))}
\end{align}
for all 
$0<\delta \leq 1/4$, 
intervals $I \subset \mathbb{R}$, 
and 
$v \in \widetilde{L^{\infty}}(I;\dB_{2,1}^{0}(\mathbb{R}^2)) \cap \widetilde{L^{\frac{2}{\delta^2}}}(I;\dot{H}^{\delta^2}(\mathbb{R}^2))$.
Then, we obtain the estimate \eqref{est:imp} with the norm replaced by $\N{\cdot}_{\delta,I}$ and the constant $C$ independent of $\delta$.
See Lemma \ref{lemm:XYZ} below.
Then, choosing the initial data $a$ so small that $\n{a}_{\dB_{2,1}^0(\mathbb{R}^2)} \leq \varepsilon_0$ 
with sufficiently small $0<\varepsilon_0 \ll 1$ independent of $\delta$
and
using the contraction mapping principle via the norm $\N{\cdot}_{\delta,I_{\delta}}$,
we may construct the the remainder part $\widetilde{u}$ satisfying 
\begin{align}
    \n{\widetilde{u}}_{\widetilde{L^{\infty}}(I_{\delta};\dB_{2,1}^{0}(\mathbb{R}^2))} \leq C\varepsilon_0.
\end{align}
Hence collecting the above estimates and for $0<\delta\leq\varepsilon_0$ and sufficiently large $M\gg 1$, 
we have 
\begin{align}
    &\n{u(t_0+k_{\delta,T}T)}_{\dB_{2,1}^0(\mathbb{R}^2)}\\
    &\quad\geq {}
    \n{u^{(2)}(t_0+k_{\delta,T}T)}_{\dB_{2,1}^{0}(\mathbb{R}^2)}
    -
    \n{u^{(1)}}_{\widetilde{L^{\infty}}(I_{\delta};\dB_{2,1}^{0}(\mathbb{R}^2))}
    -
    \n{\widetilde{u}}_{\widetilde{L^{\infty}}(I_{\delta};\dB_{2,1}^{0}(\mathbb{R}^2))}\\
    &\quad\geq {}
    cM^2
    -
    C\delta
    -
    C\varepsilon_0\\
    &\quad\geq {}
    2\varepsilon_0,
\end{align}
which completes the outline of the proof.

The following lemma provide the nonlinear estimates for the norm $\N{\cdot}_{\delta,I}$.
\begin{lemm}\label{lemm:XYZ}
For an interval $I=[t_0,t_1) \subset \mathbb{R}$,
the following statements hold.
\begin{itemize}
    \item [(1)]
    There exists an absolute positive constant $C$ such that 
    \begin{align}
        \begin{aligned}\label{nonlin_est:X}
        \N{\int_{t_0}^t e^{(t-\tau)\Delta} \mathbb{P}\div (u(\tau) \otimes v(\tau))d\tau}_{\delta,I}
        \leq{}
        C
        \N{u}_{\delta,I}\N{v}_{\delta,I}
        \end{aligned}
    \end{align}
    for all $0 \leq \delta \leq 1/4$ and $u,v \in \widetilde{L^{\infty}}(I;\dB_{2,1}^{0}(\mathbb{R}^2)) \cap \widetilde{L^{\frac{2}{\delta^2}}}(I;\dot{H}^{\delta^2}(\mathbb{R}^2))$.
    \item [(2)]
    There exists an absolute positive constant $C$ such that
    \begin{align}\label{nonlin_est:XYZ}
    &
    \n{\int_{t_0}^t e^{(t-\tau)\Delta} \mathbb{P}\div (u(\tau) \otimes v(\tau))d\tau}
    _{
    \widetilde{L^{\infty}}(I;\dB_{2,1}^{0}(\mathbb{R}^2)) 
    \cap 
    \widetilde{L^4}(I;\dB_{2,1}^{\frac{1}{2}}(\mathbb{R}^2))
    }\\
    &\quad
    \leq{}
    C
    \n{u}_{\widetilde{L^4}(I;\dB_{2,1}^{\frac{1}{2}}(\mathbb{R}^2))}
    \n{v}_{\widetilde{L^{\infty}}(I;\dB_{2,1}^{0}(\mathbb{R}^2))}
    \end{align}
    for all 
    $u \in \widetilde{L^4}(I;\dB_{2,1}^{\frac{1}{2}}(\mathbb{R}^2))$
    and
    $v\in \widetilde{L^{\infty}}(I;\dB_{2,1}^{0}(\mathbb{R}^2))$.
\end{itemize}
\end{lemm}
\begin{rem}
    We should emphasize that the positive constant $C$ appearing in \eqref{nonlin_est:X} is independent of $\delta$.
\end{rem}
\begin{proof}[Proof of Lemma \ref{lemm:XYZ}]
As (2) is obtained by Lemma \ref{lemm:nonlin-ndim}, we only prove (1).
We decompose the left hand side by the Bony decomposition as follows: 
\begin{align}
    &
    \N{\int_{t_0}^t e^{(t-\tau)\Delta} \mathbb{P}\div (u(\tau) \otimes v(\tau))d\tau}_{\delta,I}\\
    &\quad 
    \leq{}
    \N{\int_{t_0}^t e^{(t-\tau)\Delta} \mathbb{P}\div \{ T_{u_k(\tau)}v_{\ell}(\tau)+T_{v_{\ell}(\tau)}u_{k}(\tau) \}_{1 \leq k,\ell \leq 3}d\tau}_{\delta,I}\\
    &\qquad
    +
    \N{\int_{t_0}^t e^{(t-\tau)\Delta} \mathbb{P}\div \{ R(u_k(\tau),v_{\ell}(\tau)) \}_{1 \leq k,\ell \leq 3} d\tau}_{\delta,I}.
\end{align}
Here, see Appendix \ref{sec:a} for the definition of $T_fg$ and $R(f,g)$.
It follows from Lemmas \ref{lemm:max-reg} and \ref{lemm:para} that 
\begin{align}
    &
    \N{\int_{t_0}^t e^{(t-\tau)\Delta} \mathbb{P}\div \{ T_{u_k(\tau)}v_{\ell}(\tau)+T_{v_{\ell}(\tau)}u_{k}(\tau) \}_{1 \leq k,\ell \leq 3}d\tau}_{\delta,I}\\
    &\quad 
    \leq
    C
    \sum_{1 \leq k,\ell \leq 3}
    \n{T_{u_k}v_{\ell}+T_{v_{\ell}}u_{k}}_{\widetilde{L^{\infty}}(I;\dB_{2,1}^{-1}(\mathbb{R}^2))}\\
    &\qquad
    +
    \frac{C}{\delta}
    \sum_{1 \leq k,\ell \leq 3}
    \n{T_{u_k}v_{\ell}+T_{v_{\ell}}u_{k}}_{\widetilde{L^{\frac{2}{\delta^2}}}(I;\dot{H}^{\delta^2-1}(\mathbb{R}^2))}\\
    &\quad
    \leq
    C
    \n{u}_{\widetilde{L^{\infty}}(I;\dB_{2,1}^{0}(\mathbb{R}^2))}
    \n{v}_{\widetilde{L^{\infty}}(I;\dB_{2,1}^{0}(\mathbb{R}^2))}\\
    &\qquad
    +
    \frac{C}{\delta}
    \n{u}_{\widetilde{L^{\frac{2}{\delta^2}}}(I;\dot{H}^{\delta^2}(\mathbb{R}^2))}
    \n{v}_{\widetilde{L^{\infty}}(I;\dB_{2,1}^0(\mathbb{R}^2))}\\
    &\quad
    \leq 
    C
    \N{u}_{\delta,I}
    \N{v}_{\delta,I}.
\end{align}
Using Lemma \ref{lemm:para'}, we have
\begin{align}
    &
    \N{\int_{t_0}^t e^{(t-\tau)\Delta} \mathbb{P}\div \{ R(u_k(\tau),v_{\ell}(\tau)) \}_{1 \leq k,\ell \leq 3}d\tau}_{\delta,I}\\
    &\quad 
    \leq
    C
    \sum_{1 \leq k,\ell \leq 3}
    \sp{
    \n{R(u_k,v_{\ell})}_{\widetilde{L^{\frac{1}{\delta^2}}}(I;\dot{B}_{2,1}^{2\delta^2-1}(\mathbb{R}^2))}
    +
    \frac{1}{\delta}
    \n{R(u_k,v_{\ell})}_{\widetilde{L^{\frac{1}{\delta^2}}}(I;\dot{H}^{2\delta^2-1}(\mathbb{R}^2))}
    }\\
    &\quad 
    \leq
    C
    \sum_{1 \leq k,\ell \leq 3}
    \sp{
    \n{R(u_k,v_{\ell})}_{\widetilde{L^{\frac{1}{\delta^2}}}(I;\dot{B}_{1,1}^{2\delta^2}(\mathbb{R}^2))}
    +
    \frac{1}{\delta}
    \n{R(u_k,v_{\ell})}_{\widetilde{L^{\frac{1}{\delta^2}}}(I;\dB_{1,2}^{2\delta^2}(\mathbb{R}^2))}
    }\\
    &\quad 
    \leq
    \frac{C}{\delta^2}
    \n{u}_{\widetilde{L^{\frac{2}{\delta^2}}}(I;\dot{H}^{\delta^2}(\mathbb{R}^2))}
    \n{v}_{\widetilde{L^{\frac{2}{\delta^2}}}(I;\dot{H}^{\delta^2}(\mathbb{R}^2))}\\
    &\quad
    \leq 
    C
    \N{u}_{\delta,I}
    \N{v}_{\delta,I},
\end{align}
which completes the proof.
\end{proof}
Now, we provide the rigorous proof of Proposition \ref{prop:non-per}.
\begin{proof}[Proof of Proposition \ref{prop:non-per}]
Let the initial data $a \in \dB_{2,1}^0(\mathbb{R}^2)$ satisfy
\begin{align}
    \div a = 0,\qquad
    \n{a}_{\dB_{2,1}^0(\mathbb{R}^2)} 
    \leq \varepsilon.
\end{align}
Let $M \geq 10$ and $0< \varepsilon, \eta \leq 1/2$ be positive constants to be determined later.
Let $0 < \delta \leq \eta$ and $0 < T \leq 2^{\frac{1}{\delta^2}}$.
For any $\mathbb{R}^2$-valued {divergence-free} $T$-periodic function ${h} \in \widetilde{C}(\mathbb{R};\dB_{p,1}^{\frac{2}{p}-3}(\mathbb{R}^2)) \cap \widetilde{L^{\infty}}(\mathbb{R};\dB_{p,2}^{\frac{2}{p}-3+\delta^2}(\mathbb{R}^2))$ with  
\begin{align}
    \n{{h}}_{\widetilde{L^{\infty}}(\mathbb{R};\dB_{p,1}^{\frac{2}{p}-3}(\mathbb{R}^2))}
    +
    \n{{h}}_{\widetilde{L^{\infty}}(\mathbb{R};\dB_{p,1}^{\frac{2}{p}-3+\delta^2}(\mathbb{R}^2))}\leq 1,
\end{align}
we define the external force as 
\begin{align}\label{ex:f}
    \begin{split}
    f_{\delta}(t,x)
    :={}&
    \eta 
    \delta\Delta g(x) + \eta^2
    \delta 
    {h}(t,x),\\
    g(x)
    :={}&
    \nabla^{\perp}
    \left( 
    \psi(x)
    \cos(Mx_1)
    \right),
    \end{split}
\end{align}
where 
the function $\psi \in \mathscr{S}(\mathbb{R}^2)$ satisfy that $\widehat{\psi}$ is radial symmetric and 
\begin{align}
    0 \leq \widehat{\psi}(\xi) \leq 1,
    \qquad
    \widehat{\psi}(\xi)
    =
    \begin{cases}
        {1} & (|\xi| \leq 1 ),\\
        {0} & (|\xi| \geq 2).
    \end{cases}
\end{align}
We note that $f_{\delta}$ is $\mathbb{R}^3$-valued, $T$-periodic, and divergence free.
Using Lemma \ref{lemm:max-reg} and $\supp \widehat{g} \subset \{ \xi \in \mathbb{R}^2\ ;\ M-2 \leq | \xi | \leq M+2  \}$, we have
\begin{align}
    &
    \n{e^{t\Delta}a}_{\widetilde{L^{\infty}}(I_{\delta};\dB_{2,1}^{0}(\mathbb{R}^2)) \cap \widetilde{L^4}(I_{\delta};\dB_{2,1}^{\frac{1}{2}}(\mathbb{R}^2))}
    \leq
    C_0
    \n{a}_{\dB_{2,1}^0(\mathbb{R}^2)},\label{est:a}\\
    &
    \begin{aligned}
    \n{f_{\delta}}_{\widetilde{L^{\infty}}(\mathbb{R};\dB_{p,1}^{\frac{2}{p}-3}(\mathbb{R}^2))}
    \leq{}&
    C\eta\delta\n{g}_{\dB_{p,1}^{\frac{2}{p}-1}(\mathbb{R}^2)}
    +
    \eta^2\delta\n{h}_{\widetilde{L^{\infty}}(\mathbb{R};\dB_{p,1}^{\frac{2}{p}-3}(\mathbb{R}^2))}\\
    \leq{}&
    CM^{\frac{2}{p}}\n{\psi}_{L^p(\mathbb{R}^2)}\eta\delta
    +\eta^2\delta\\
    \leq{}&
    C_0M^{\frac{2}{p}}\eta\delta
    \end{aligned}
\end{align}
for some positive constant $C_0=C_0(p,\n{\psi}_{L^p(\mathbb{R}^2)})$.

We set $k_{\delta,T} \in \mathbb{N}$ and $I_{\delta} \subset \mathbb{R}$ as  
\begin{align}
    \frac{2^{\frac{1}{\delta^2}}}{T}
    \leq
    k_{\delta,T}
    <
    \frac{2^{\frac{1}{\delta^2}}}{T}+1,
    \qquad
    I_{\delta}:=[t_0,t_0+k_{\delta,T}T].
\end{align}
Here, we note that it holds 
\begin{align}\label{k_delta_T}
    2 \leq (k_{\delta,T}T)^{\delta^2} < 4.
\end{align}
We define
\begin{align}
    u_{\delta}^{(1)}(t)
    :={}&
    \int_{t_0}^t e^{(t-\tau)\Delta}\mathbb{P}{f_{\delta}}(\tau)d\tau
    =
    \int_{t_0}^t e^{(t-\tau)\Delta}{f_{\delta}}(\tau)d\tau,\\
    u_{\delta}^{(2)}(t)
    :={}&
    -
    \int_{t_0}^t e^{(t-\tau)\Delta}\mathbb{P}\div\left(u_{\delta}^{(1)}(\tau) \otimes u_{\delta}^{(1)}(\tau)\right)d\tau,\\
    u_{\delta}^{(3)}(t)
    :={}&
    -
    \int_{t_0}^t 
    e^{(t-\tau)\Delta}
    \mathbb{P}
    \div
    \left(
    u_{\delta}^{(1)}(\tau) \otimes u_{\delta}^{(2)}(\tau)
    +
    u_{\delta}^{(2)}(\tau) \otimes u_{\delta}^{(1)}(\tau)
    \right.
    \\
    &\quad
    \left.
    {}
    +
    u_{\delta}^{(2)}(\tau) \otimes u_{\delta}^{(2)}(\tau)\right)d\tau.
\end{align}
and consider the following integral equation:
\begin{align}
    \widetilde{u}(t)
    ={}&
    e^{t\Delta}a
    +
    u_{\delta}^{(3)}(t)\\
    &-
    \sum_{m=1}^2
    \int_{t_0}^t 
    e^{(t-\tau)\Delta}
    \mathbb{P}
    \div
    \left(
    u_{\delta}^{(m)}(\tau) \otimes \widetilde{u}(\tau)
    +
    \widetilde{u}(\tau) \otimes u_{\delta}^{(m)}(\tau)
    \right)d\tau\\
    &-
    \int_{t_0}^t 
    e^{(t-\tau)\Delta}
    \mathbb{P}
    \div
    \left(
    \widetilde{u}(\tau) \otimes \widetilde{u}(\tau)
    \right)d\tau. \label{eq:IE_rem}
\end{align}
We note that once we establish a solution $\widetilde{u}_{\delta}[a]$ to \eqref{eq:IE_rem}, we obtain the mild solution to \eqref{eq:NS_I} by
$u_{\delta}[a](t) := u_{\delta}^{(1)}(t) + u_{\delta}^{(2)}(t) + \widetilde{u}_{\delta}[a](t)$.

For the estimates of $u_{\delta}^{(1)}$, 
we decompose it as
\begin{align}
    u_{\delta}^{(1)}(t)
    &=
    \eta
    \delta
    \Delta
    \int_{t_0}^t
    e^{(t-\tau)\Delta}
    g
    d\tau
    +
    \int_{t_0}^t
    e^{(t-\tau)\Delta}
    {h}(\tau)d\tau\\
    &=
    -
    \eta
    \delta
    g 
    +
    \eta
    \delta
    e^{(t-t_0)\Delta}g
    +
    \eta^2\delta 
    \int_{t_0}^t
    e^{(t-\tau)\Delta}
    {h}(\tau)d\tau\\
    &=:
    u_{\delta}^{(1;1)}
    +
    u_{\delta}^{(1;2)}(t)
    +
    u_{\delta}^{(1;3)}(t).
\end{align}
Then, 
it follows from Lemma \ref{lemm:max-reg} {and the Bernstein inequality} that 
\begin{align}
    &
    \begin{aligned}\label{est:v_1_1_XYZ}
    \n{u_{\delta}^{(1;1)}}_{\widetilde{L^{\infty}}(I_{\delta};\dB_{2,1}^{0}(\mathbb{R}^2))}
    =
    \eta\delta
    \n{g}_{\dB_{2,1}^{0}(\mathbb{R}^2)}
    \leq{}
    C
    {M}
    \n{\psi}_{L^2(\mathbb{R}^2)}
    \eta
    \delta
    \leq
    C_1
    M
    \eta 
    \delta,
    \end{aligned}\\
    &
    \begin{aligned}\label{est:v_1_2_XYZ}
    \n{u_{\delta}^{(1;2)}}_{\widetilde{L^{\infty}}(I_{\delta};\dB_{2,1}^{0}(\mathbb{R}^2))\cap \widetilde{L^{\frac{2}{\delta^2}}}(I_{\delta};\dot{H}^{\delta^2}(\mathbb{R}^2))}
    \leq{}
    C
    \eta\delta
    \n{g}_{\dB_{2,1}^0(\mathbb{R}^2)}
    \leq{}
    C_1
    M
    \eta
    \delta,
    \end{aligned}\\
    &
    \begin{aligned}\label{est:v_1_3_X}
    \n{u_{\delta}^{(1;3)}}_{\widetilde{L^{\infty}}(I_{\delta};\dB_{2,1}^{0}(\mathbb{R}^2))}
    \leq{}
    C
    \eta^2\delta
    \n{{h}}_{\widetilde{L^{\infty}}(I_{\delta};{\dB_{p,1}^{\frac{2}{p}-3}}(\mathbb{R}^2))}
    \leq{}
    C_1\eta^2\delta
    \end{aligned}
\end{align}
and
\begin{align}
    &
    \begin{aligned}\label{est:v_1_1_Y}
    \n{u_{\delta}^{(1;1)}}_{\widetilde{L^{\frac{2}{\delta^2}}}(I_{\delta};\dot{H}^{\delta^2}(\mathbb{R}^2))}
    \leq{}&
    C
    \eta\delta
    {(k_{\delta,T}T)^{\frac{\delta^2}{2}}}
    \n{g}_{\dot{H}^{\delta^2}(\mathbb{R}^2)}
    \leq{}
    C_1
    M^2
    \eta 
    \delta,
    \end{aligned}\\
    &
    \begin{aligned}\label{est:v_1_3_Y}
    \n{u_{\delta}^{(1;3)}}_{\widetilde{L^{\frac{2}{\delta^2}}}(I_{\delta};\dot{H}^{\delta^2}(\mathbb{R}^2))}
    \leq{}&
    C
    \eta^2
    \delta
    \n{{h}}_{\widetilde{L^{\frac{2}{\delta^2}}}(I_{\delta};\dot{H}^{-2+\delta^2}(\mathbb{R}^2))}\\
    \leq{}&
    C
    \eta^2
    \delta
    {(k_{\delta,T}T)^{\frac{\delta^2}{2}}}
    \n{{h}}_{\widetilde{L^{\infty}}(I_{\delta};\dB_{p,2}^{\frac{2}{p}-3+\delta^2}(\mathbb{R}^2))}
    \leq{}
    C_1
    \eta^2\delta
    \end{aligned}
\end{align}
for some positive constant $C_1$.
Thus, we have
\begin{align}\label{est:u_1_XY}
    {\n{u_{\delta}^{(1)}}_{\widetilde{L^{\infty}}(I_{\delta};\dB_{2,1}^{0}(\mathbb{R}^2))}
    \leq
    3C_1M^2\eta\delta,} \quad
    \N{u_{\delta}^{(1)}}_{\delta,I_{\delta}}
    \leq 
    {6C_1M^2\eta}.
\end{align}
For the estimates for $u_{\delta}^{(2)}$,
we decompose it as 
\begin{align}
    u_{\delta}^{(2)}(t)
    &=
    \sum_{k,\ell=1}^3
    u_{\delta}^{(2;k,\ell)}(t)\\
    &=
    u_{\delta}^{(2;1,1)}(t)
    +
    \sum_{(k,\ell) \in \{ 1,3 \}^2\setminus \{ (1,1)\}}
    u_{\delta}^{(2;k,\ell)}(t)
    +
    \sum_{(k,\ell) \in \{ 1,2,3 \}^2 \setminus \{ 1,3 \}^2}
    u_{\delta}^{(2;k,\ell)}(t)\\
    &=:
    u_{\delta}^{(2;1)}(t)
    +
    u_{\delta}^{(2;2)}(t)
    +
    u_{\delta}^{(2;3)}(t),
\end{align}
where we have set
\begin{align}
    u_{\delta}^{(2;k,\ell)}(t)
    :=
    -\int_{t_0}^t
    e^{(t-\tau)\Delta}
    \mathbb{P}\div\left(u_{\delta}^{(1;k)}(\tau) \otimes u_{\delta}^{(1;\ell)}(\tau)\right)
    d\tau.
\end{align}
It follows from  Lemma \ref{lemm:XYZ} that
\begin{align}
    &
    \begin{aligned}\label{est:u_2_1_XY}
    \N{u_{\delta}^{(2;1)}}_{\delta,I_{\delta}}
    \leq{}
    C
    \N{u_{\delta}^{(1;1)}}_{\delta,I_{\delta}}^2
    \leq{}
    C_2M^4\eta^2,
    \end{aligned}\\
    &
    \begin{aligned}\label{est:u_2_2_XY}
    \N{u_{\delta}^{(2;2)}}_{\delta,I_{\delta}}
    \leq{}&
    C
    \sum_{ (k,\ell) \in \{ 1,3 \}^2 \setminus \{ (1,1) \}}
    \N{u_{\delta}^{(1;k)}}_{\delta,I_{\delta}}
    \N{u_{\delta}^{(1;\ell)}}_{\delta,I_{\delta}}\\
    \leq{}&
    C_2M^2\eta^3
    \end{aligned}\\
    &
    \begin{aligned}\label{est:u_2_3_XY}
    &
    \n{u_{\delta}^{(2;3)}}_{\widetilde{L^{\infty}}(I_{\delta};\dB_{2,1}^{0}(\mathbb{R}^2))\cap \widetilde{L^{\frac{2}{\delta^2}}}(I_{\delta};\dot{H}^{\delta^2}(\mathbb{R}^2))}\\
    &\quad 
    \leq{}
    C
    \n{u_{\delta}^{(1;2)}}_{\widetilde{L^4}(I_{\delta};\dB_{2,1}^{\frac{1}{2}}(\mathbb{R}^2))}
    \sum_{\ell=1}^3
    \n{u_{\delta}^{(1;\ell)}}_{\widetilde{L^{\infty}}(I_{\delta};\dB_{2,1}^{0}(\mathbb{R}^2))}\\
    &\quad
    \leq{}
    C_2M^2\eta^2\delta^2
    \end{aligned}
\end{align}
for some positive constant $C_2$.
Thus, we have\footnote{{To obtain the estimate \eqref{est:u_2_XY}, we do not need the decomposition and it suffices to apply Lemma \ref{lemm:XYZ} to the definition of $u_\delta^{(2)}$. 
However, constructing the estimate from below in \eqref{est:below}, we use this decompositions $u^{(2,k)}$; therefore we prepare them here.}}
\begin{align}\label{est:u_2_XY}
    \N{u_{\delta}^{(2)}}_{\delta,I_{\delta}}
    \leq 
    3C_2M^4\eta^2.
\end{align}
For the estimates of $u_{\delta}^{(3)}$, 
we rewrite it as
\begin{align}
    u_{\delta}^{(3)}(t)
    ={}
    -
    \sum_{(i,j) \in \{1,2\}^2 \setminus \{ (1,1) \}}
    \int_{t_0}^t 
    e^{(t-\tau)\Delta}
    \mathbb{P}
    \div
    \left(
    u_{\delta}^{(i)}(\tau) \otimes u_{\delta}^{(j)}(\tau)
    \right)d\tau,
\end{align}
by Lemma \ref{lemm:XYZ}, there holds
\begin{align}\label{est:u_3_XY}
    \N{u_{\delta}^{(3)}}_{\delta,I_{\delta}}
    \leq{}&
    C
    \sum_{(i,j) \in \{1,2\}^2 \setminus \{ (1,1) \}}
    \N{u_{\delta}^{(i)}}_{\delta,I_{\delta}}
    \N{u_{\delta}^{(j)}}_{\delta,I_{\delta}}
    \leq{}
    C_3M^8\eta^3.
\end{align}
for some positive constant $C_3$.
Now, we construct a solution $\widetilde{u}_{\delta}$ to the equation \eqref{eq:IE_rem}.
To this end, we define a complete metric space $(S_{\delta},d_{S_{\delta}})$ by
\begin{align}
    &
    S_{\delta}
    :={}
    \left\{
    u=u'+u'' \in \widetilde{C}(I_{\delta};\dB_{2,1}^0(\mathbb{R}^2))\ ;\ 
    u' \in S_{\delta}',
    u'' \in S_{\delta}''
    \right\},\\
    &
    \begin{aligned}
    &
    d_{S_{\delta}}(u,v)
    :={}
    \inf_{\substack{u=u'+u''\\ v= v'+v'' \\ u',v' \in S_{\delta}'\\ u'',v'' \in S_{\delta}''}}
    \left(
    \n{u'-v'}_{\widetilde{L^{\infty}}(I_{\delta};\dB_{2,1}^{0}(\mathbb{R}^2)) \cap \widetilde{L^4}(I_{\delta};\dB_{2,1}^{\frac{1}{2}}(\mathbb{R}^2))}
    +
    \N{u''-v''}_{\delta,I_{\delta}}
    \right),
    \end{aligned}
\end{align}
where 
\begin{align}
    S_{\delta}'
    :={}&
    \left\{
    \begin{aligned}
    u' 
    \in{}& 
    \widetilde{L^{\infty}}(I_{\delta};\dB_{2,1}^{0}(\mathbb{R}^2)) \\
    &
    \cap 
    \widetilde{L^4}(I_{\delta};\dB_{2,1}^{\frac{1}{2}}(\mathbb{R}^2))
    \end{aligned}
    \ ;\ 
    \n{u'}_{\widetilde{L^{\infty}}(I_{\delta};\dB_{2,1}^{0}(\mathbb{R}^2))\cap \widetilde{L^4}(I_{\delta};\dB_{2,1}^{\frac{1}{2}}(\mathbb{R}^2))}
    \leq 
    4C_0 \varepsilon
    \right\},\\
    S_{\delta}''
    :={}&
    \left\{
    u'' \in \widetilde{L^{\infty}}(I_{\delta};\dB_{2,1}^{0}(\mathbb{R}^2)) \cap \widetilde{L^{\frac{2}{\delta^2}}}(I_{\delta};\dot{H}^{\delta^2}(\mathbb{R}^2))
    \ ;\ 
    \N{u''}_{\delta,I_{\delta}}
    \leq 
    4C_3 M^8\eta^3
    \right\},
\end{align}
and let us consider a map
\begin{align}
    \Phi_{\delta}[u](t)
    ={}&
    e^{t\Delta}a
    +
    u_{\delta}^{(3)}(t)\\
    &-
    \sum_{m=1}^2
    \int_{t_0}^t 
    e^{(t-\tau)\Delta}
    \mathbb{P}
    \div
    \left(
    u_{\delta}^{(m)}(\tau) \otimes u(\tau)
    +
    u(\tau) \otimes u_{\delta}^{(m)}(\tau)
    \right)d\tau\\
    &-
    \int_{t_0}^t 
    e^{(t-\tau)\Delta}
    \mathbb{P}
    \div
    \left(
    u(\tau) \otimes u(\tau)
    \right)d\tau,\qquad u \in S_{\delta}.
\end{align}
For any $u=u'+u'' \in S_{\delta}$ with $u' \in S_{\delta}'$ and $u'' \in S_{\delta}''$, 
we decompose $\Phi_{\delta}[u]$ as 
\begin{align}
    \Phi_{\delta}[u](t)
    ={}&
    \Phi_{\delta}'[u',u''](t)
    +
    \Phi_{\delta}''[u''](t),
\end{align}
where we have set
\begin{align}
    &
    \begin{aligned}
    \Phi_{\delta}'[u',u''](t)
    :={}
    e^{t\Delta}a
    &
    -
    \sum_{m=1}^2
    \int_{t_0}^t 
    e^{(t-\tau)\Delta}
    \mathbb{P}
    \div
    \left(
    u_{\delta}^{(m)}(\tau) \otimes u'(\tau)
    +
    u'(\tau) \otimes u_{\delta}^{(m)}(\tau)
    \right)d\tau
    \\
    &-
    \int_{t_0}^t 
    e^{(t-\tau)\Delta}
    \mathbb{P}
    \div
    \left(
    u'(\tau) \otimes u''(\tau)
    +
    u''(\tau) \otimes u'(\tau)
    \right)d\tau\\
    &-
    \int_{t_0}^t 
    e^{(t-\tau)\Delta}
    \mathbb{P}
    \div
    \left(
    u'(\tau) \otimes u'(\tau)
    \right)d\tau,
    \end{aligned}\\
    &
    \begin{aligned}
    \Phi_{\delta}''[u''](t)
    :={}
    u_{\delta}^{(3)}(t)
    &
    -
    \sum_{m=1}^2
    \int_{t_0}^t 
    e^{(t-\tau)\Delta}
    \mathbb{P}
    \div
    \left(
    u_{\delta}^{(m)}(\tau) \otimes u''(\tau)
    +
    u''(\tau) \otimes u_{\delta}^{(m)}(\tau)
    \right)d\tau\\
    &-
    \int_{t_0}^t 
    e^{(t-\tau)\Delta}
    \mathbb{P}
    \div
    \left(
    u''(\tau) \otimes u''(\tau)
    \right)d\tau.
    \end{aligned}
\end{align}
It follows from Lemma \ref{lemm:XYZ}, \eqref{est:a}, \eqref{est:u_1_XY}, \eqref{est:u_2_XY}, and \eqref{est:u_3_XY} that 
\begin{align}
    &
    \n{\Phi_{\delta}'[u',u'']}_{\widetilde{L^{\infty}}(I_{\delta};\dB_{2,1}^{0}(\mathbb{R}^2)) \cap \widetilde{L^4}(I_{\delta};\dB_{2,1}^{\frac{1}{2}}(\mathbb{R}^2))}\\
    &\quad 
    \leq{}
    C_0
    \n{a}_{\dB_{2,1}^0(\mathbb{R}^2)}
    +
    C
    \sum_{m=1}^2
    \n{u_{\delta}^{(m)}}_{\widetilde{L^{\infty}}(I_{\delta};\dB_{2,1}^{0}(\mathbb{R}^2))}
    \n{u'}_{\widetilde{L^4}(I_{\delta};\dB_{2,1}^{\frac{1}{2}}(\mathbb{R}^2))}\\
    &\qquad
    +
    C
    \n{u''}_{\widetilde{L^{\infty}}(I_{\delta};\dB_{2,1}^{0}(\mathbb{R}^2))}
    \n{u'}_{\widetilde{L^4}(I_{\delta};\dB_{2,1}^{\frac{1}{2}}(\mathbb{R}^2))}
    +
    C
    \n{u'}_{\widetilde{L^4}(I_{\delta};\dB_{2,1}^{\frac{1}{2}}(\mathbb{R}^2))}^2\\
    &\quad
    \leq{}
    C_0
    \varepsilon
    +
    C_4
    M^4
    \eta^2
    \varepsilon
    +
    C_4
    M^8
    \eta^3
    \varepsilon
    +
    C_4
    \varepsilon^2
\end{align}
and
\begin{align}
    \N{\Phi_{\delta}''[u'']}_{\delta,I_{\delta}}
    \leq{}&
    \N{v_{\delta}^{(3)}}_{\delta,I_{\delta}}
    +
    C
    \sum_{m=1}^2
    \N{u_{\delta}^{(m)}}_{\delta,I_{\delta}}
    \N{u''}_{\delta,I_{\delta}}
    +
    C
    \N{u''}_{\delta,I_{\delta}}^2\\
    \leq{}&
    C_3
    M^8\eta^3
    +
    C_4
    (M^4\eta^2)
    M^8\eta^3
    +
    C_4
    (M^8\eta^3)^2
\end{align}
for some positive constant $C_4$.
Here, we have used 
\begin{align}
    \sum_{m=1}^2
    \N{u_{\delta}^{(m)}}_{\delta,I_{\delta}}
    \leq
    CM^2\eta\delta + CM^4\eta^2
    \leq
    CM^4\eta^2,
\end{align}
which is implied by \eqref{est:u_1_XY}, \eqref{est:u_2_XY}, and $\delta \leq \eta$.
Let $0<\eta_1,\varepsilon_1 \leq 1/2$ satisfy
\begin{align}
    \max\mp{
    M^4
    \eta_1^2,
    \varepsilon_1,
    M^8
    \eta_1^3}
    \leq 
    \min\mp{\frac{C_0}{3C_4}, \frac{C_3}{3C_4}}
\end{align}
and assume $0 < \eta \leq \eta_1$ and $0<\varepsilon \leq \varepsilon_1$ in the following of this proof.
Then, we see that for every $u\in S_{\delta}$, 
\begin{align}\label{Phi'-Phi''}
    \begin{split}
    &\n{\Phi_{\delta}'[u',u'']}_{\widetilde{L^{\infty}}(I_{\delta};\dB_{2,1}^{0}(\mathbb{R}^2)) \cap \widetilde{L^4}(I_{\delta};\dB_{2,1}^{\frac{1}{2}}(\mathbb{R}^2))}
    \leq{}
    2C_0\varepsilon,\\
    &
    \N{\Phi_{\delta}''[u'']}_{\delta,I_{\delta}}
    \leq{}
    2C_3
    M^8\eta^3,
    \end{split}
\end{align}
which implies $\Phi_{\delta}'[u',u''] \in S_{\delta}'$ and $\Phi_{\delta}''[u''] \in S_{\delta}''$.
Thus, we have $\Phi_{\delta}[u] \in S_{\delta}$ for all $u \in S_{\delta}$.
For $u=u'+u'',v=v'+v'' \in S_{\delta}$ with $u',v' \in S_{\delta}'$ and $u'',v'' \in S_{\delta}''$, 
since it holds
\begin{align}
    &\Phi_{\delta}'[u',u''](t)
    -
    \Phi_{\delta}'[v',v''](t)\\
    &\quad 
    ={}
    -
    \sum_{m=1}^2
    \int_{t_0}^t 
    e^{(t-\tau)\Delta}
    \mathbb{P}
    \div
    \left(
    u_{\delta}^{(m)}(\tau) \otimes (u'(\tau)-v'(\tau))
    \right)d\tau
    \\
    &\qquad
    -
    \sum_{m=1}^2
    \int_{t_0}^t 
    e^{(t-\tau)\Delta}
    \mathbb{P}
    \div
    \left(
    (u'(\tau)-v'(\tau)) \otimes u_{\delta}^{(m)}(\tau)
    \right)d\tau
    \\
    &\qquad
    -
    \int_{t_0}^t 
    e^{(t-\tau)\Delta}
    \mathbb{P}
    \div
    \left(
    (u'(\tau)-v'(\tau)) \otimes u''(\tau)
    +
    v'(\tau) \otimes (u''(\tau)-v''(\tau))
    \right)d\tau\\
    &\qquad
    -
    \int_{t_0}^t 
    e^{(t-\tau)\Delta}
    \mathbb{P}
    \div
    \left(
    (u''(\tau)-v''(\tau)) \otimes u'(\tau)
    +
    v''(\tau) \otimes (u'(\tau)-v'(\tau))
    \right)d\tau\\
    &\qquad
    -
    \int_{t_0}^t 
    e^{(t-\tau)\Delta}
    \mathbb{P}
    \div
    \left(
    (u'(\tau)-v'(\tau)) \otimes u'(\tau)
    +
    v'(\tau) \otimes (u'(\tau)-v'(\tau))
    \right)d\tau
\end{align}
and
\begin{align}
    &\Phi_{\delta}''[u''](t)
    -
    \Phi_{\delta}''[v''](t)\\
    &\quad 
    ={}
    -
    \sum_{m=1}^2
    \int_{t_0}^t 
    e^{(t-\tau)\Delta}
    \mathbb{P}
    \div
    \left(
    u_{\delta}^{(m)}(\tau) \otimes (u''(\tau)-v''(\tau))
    \right)d\tau
    \\
    &{}\qquad
    -
    \sum_{m=1}^2
    \int_{t_0}^t 
    e^{(t-\tau)\Delta}
    \mathbb{P}
    \div
    \left(
    (u''(\tau)-v''(\tau)) \otimes u_{\delta}^{(m)}(\tau)
    \right)d\tau\\
    &{}\qquad
    -
    \int_{t_0}^t 
    e^{(t-\tau)\Delta}
    \mathbb{P}
    \div
    \left(
    (u''(\tau)-v''(\tau)) \otimes u''(\tau)
    +
    v''(\tau) \otimes (u''(\tau)-v''(\tau))
    \right)d\tau,
\end{align}
we have 
\begin{align}
    &
    \n{\Phi_{\delta}'[u',u'']
    -
    \Phi_{\delta}'[v',v'']}_{\widetilde{L^{\infty}}(I_{\delta};\dB_{2,1}^{0}(\mathbb{R}^2))\cap \widetilde{L^4}(I_{\delta};\dB_{2,1}^{\frac{1}{2}}(\mathbb{R}^2))}\\
    &\quad 
    \leq{}
    C
    \sum_{m=1}^2
    \n{u_{\delta}^{(m)}}_{\widetilde{L^{\infty}}(I_{\delta};\dB_{2,1}^{0}(\mathbb{R}^2))}
    \n{u'-v'}_{\widetilde{L^4}(I_{\delta};\dB_{2,1}^{\frac{1}{2}}(\mathbb{R}^2))}\\
    &\qquad
    +
    C
    \sp{
    \n{u'}_{\widetilde{L^4}(I_{\delta};\dB_{2,1}^{\frac{1}{2}}(\mathbb{R}^2))}
    +
    \n{v'}_{\widetilde{L^4}(I_{\delta};\dB_{2,1}^{\frac{1}{2}}(\mathbb{R}^2))}
    }
    \n{u''-v''}_{\widetilde{L^{\infty}}(I_{\delta};\dB_{2,1}^{0}(\mathbb{R}^2))}\\
    &\qquad
    +
    C
    \sp{
    \n{u''}_{\widetilde{L^{\infty}}(I_{\delta};\dB_{2,1}^{0}(\mathbb{R}^2))}
    +
    \n{v''}_{\widetilde{L^{\infty}}(I_{\delta};\dB_{2,1}^{0}(\mathbb{R}^2))}
    }
    \n{u'-v'}_{\widetilde{L^4}(I_{\delta};\dB_{2,1}^{\frac{1}{2}}(\mathbb{R}^2))}\\
    &\qquad
    +
    C
    \sp{
    \n{u'}_{\widetilde{L^{\infty}}(I_{\delta};\dB_{2,1}^{0}(\mathbb{R}^2))}
    +
    \n{v'}_{\widetilde{L^{\infty}}(I_{\delta};\dB_{2,1}^{0}(\mathbb{R}^2))}
    }
    \n{u'-v'}_{\widetilde{L^4}(I_{\delta};\dB_{2,1}^{\frac{1}{2}}(\mathbb{R}^2))}\\
    &\quad
    \leq 
    C_5
    \sp{\varepsilon+M^8\eta}
    \sp{
    \n{u'-v'}_{\widetilde{L^{\infty}}(I_{\delta};\dB_{2,1}^{0}(\mathbb{R}^2)) \cap \widetilde{L^4}(I_{\delta};\dB_{2,1}^{\frac{1}{2}}(\mathbb{R}^2))}
    +
    \N{u''-v''}_{\delta,I_{\delta}}
    }
\end{align}
and
\begin{align}
    &
    \N{
    {\Phi_{\delta}''[u'']}
    -
    {\Phi_{\delta}''[v'']}
    }_{\delta,I_{\delta}}\\
    &\quad 
    \leq{}
    C
    \sum_{m=1}^2
    {
    \N{u_{\delta}^{(m)}}_{\delta,I_{\delta}}
    \N{u''-v''}_{\delta,I_{\delta}}
    }\\
    &\qquad
    +
    C
    \sp{
    \N{u''}_{\delta,I_{\delta}}
    +
    \N{v''}_{\delta,I_{\delta}}
    }
    \N{u''-v''}_{\delta,I_{\delta}}\\
    &\quad
    \leq 
    C_5M^8\eta
    \sp{
    \n{u'-v'}_{\widetilde{L^{\infty}}(I_{\delta};\dB_{2,1}^{0}(\mathbb{R}^2)) \cap \widetilde{L^4}(I_{\delta};\dB_{2,1}^{\frac{1}{2}}(\mathbb{R}^2))}
    +
    \N{u''-v''}_{\delta,I_{\delta}}
    }
\end{align}
for some positive constant $C_5$.
Thus, it holds
\begin{align}
    &
    d_{S_{\delta}}(\Phi_{\delta}[u],\Phi_{\delta}[v])\\
    &\quad 
    \leq{}
    \inf_{\substack{u=u'+u'',\ v= v'+v'' \\ u',v' \in S_{\delta}',\ u'',v'' \in S_{\delta}''}}
    \left(
    \n{
    \Phi_{\delta}'[u',u'']
    -
    \Phi_{\delta}'[v',v'']}_{\widetilde{L^{\infty}}(I_{\delta};\dB_{2,1}^{0}(\mathbb{R}^2))\cap \widetilde{L^4}(I_{\delta};\dB_{2,1}^{\frac{1}{2}}(\mathbb{R}^2))}\right.\\
    &\left.\qquad \qquad \qquad \qquad \qquad \qquad{}
    +
    \N{
    {\Phi_{\delta}''[u'']}
    -
    {\Phi_{\delta}''[v'']}}_{\delta,I_{\delta}
    }
    \right)\\
    &\quad 
    \leq{}
    C_5
    \sp{\varepsilon+2M^8\eta}
    d_{S_{\delta}}(u,v),
\end{align}
where the first estimate above is ensured by the fact   
$
\Phi_{\delta}'[u',u'']
-
\Phi_{\delta}'[v',v'']
\in 
S_{\delta}'
$
and 
$
{\Phi_{\delta}''[u'']}
-
{\Phi_{\delta}''[v'']}
\in
S_{\delta}''
$,
which are implied by \eqref{Phi'-Phi''}.
Let $0< \varepsilon_2, \eta_2 \leq 1/2$ satisfy 
\begin{align}
    C_5
    \sp{\varepsilon_2+2M^8\eta_2}
    \leq 
    \frac{1}{2}.
\end{align}
In the following of this proof, we assume $0 < \varepsilon \leq \min\{\varepsilon_1,\varepsilon_2\}$ and $0 < \eta \leq \min \{ \eta_1, \eta_2 \}$.
Then, we see that 
\begin{align}
    d_{S_{\delta}}(\Phi_{\delta}[u],\Phi_{\delta}[v])
    \leq{}&
    \frac{1}{2}
    d_{S_{\delta}}(u,v)
\end{align}
for all $u,v \in S_{\delta}$.
Hence, the Banach fixed point theorem yields the unique existence of $\widetilde{u}_{\delta}[a] \in S_{\delta}$ satisfying $\widetilde{u}_{\delta}[a] = \Phi_{\delta}[\widetilde{u}_{\delta}[a]]$, which implies $\widetilde{u}_{\delta}[a]$ is a solution to \eqref{eq:IE_rem}.
We note that $\widetilde{u}_{\delta}[a] \in S_{\delta}$ implies 
\begin{align}\label{est:rem}
    \n{\widetilde{u}_{\delta}[a]}_{\widetilde{L^{\infty}}(I_{\delta};\dB_{2,1}^{0}(\mathbb{R}^2))}
    \leq
    4C_0\varepsilon + 4C_3M^8\eta^3.
\end{align}
Moreover, 
$u_{\delta}[a]:=u_{\delta}^{(1)}+u_{\delta}^{(2)}+\widetilde{u}_{\delta}[a]$ is a solution to \eqref{eq:IE_u}.

Finally, we establish the second estimate in \eqref{est:u}.
To this end, we focus on $u_{\delta}^{(2;1)}$. 
As $g=g(x)$ is time-independent, there holds
\begin{align}
    u_{\delta}^{(2;1)}(t_0+k_{\delta,T}T)
    &=
    -
    \eta^2\delta^2
    \int_{t_0}^{t_0+k_{\delta,T}T}
    e^{(t_0+k_{\delta,T}T-\tau)\Delta}
    \mathbb{P}
    \div 
    (g \otimes g)d\tau\\
    &=
    \eta^2\delta^2
    \sp{
    1-e^{k_{\delta,T}T\Delta}
    }
    (-\Delta)^{-1}
    \mathbb{P}
    \div
    \sp{ g \otimes g}.
\end{align}
Then, we see that
\begin{align}
    &
    \n{u_{\delta}^{(2;1)}(t_0+k_{\delta,T}T)}_{\dB_{2,1}^0(\mathbb{R}^2)}\\
    &\quad
    \geq
    \eta^2\delta^2
    \sum_{-\frac{1}{2\delta^2} \leq j \leq {-2}}
    \n{\Delta_j\sp{
    1-e^{k_{\delta,T}T\Delta}
    }
    (-\Delta)^{-1}
    \mathbb{P}
    \div
    \sp{ g \otimes g}}_{L^2(\mathbb{R}^2)}\\
    &\quad
    \geq
    c\eta^2\delta^2
    \sum_{-\frac{1}{2\delta^2} \leq j \leq {-2}}
    \sp{
    1-e^{-\frac{1}{4}2^{2j}k_{\delta,T}T}
    }
    \n{
    \Delta_j
    (-\Delta)^{-1}
    \mathbb{P}
    \div
    \sp{ g \otimes g}}_{L^2(\mathbb{R}^2)}\\
    &\quad
    \geq
    c\eta^2\delta^2
    \sum_{-\frac{1}{2\delta^2} \leq j \leq {-2}}
    \n{
    \Delta_j
    (-\Delta)^{-1}
    \mathbb{P}
    \div
    \sp{ g \otimes g}}_{L^2(\mathbb{R}^2)},
\end{align}
where we have used the estimate
\begin{align}
    1 - e^{-\frac{1}{4}2^{2j}k_{\delta,T}T}
    \geq
    1 - e^{-\frac{1}{4}2^{-\frac{1}{\delta^2}}k_{\delta,T}T}
    \geq
    1 - e^{-\frac{1}{4}2^{-\frac{1}{\delta^2}}2^{\frac{1}{\delta^2}}}
    =
    1 - e^{-\frac{1}{4}}
\end{align}
which is implied by $-1/(2\delta^2) \leq j \leq {-2}$ and \eqref{k_delta_T}.
We consider the estimate for $\Delta_j\sp{ g \otimes g}$.
It follows from \cite{Fujii-pre-1}*{Lemmas 2.1 and 2.4} that for $-1/(2\delta^2) \leq j \leq {-2}$,
\begin{align}\label{lower-est-1}
    \Delta_j{\div}\sp{ g \otimes g}
    =
    \frac{M^2}{2}
    \Delta_j  
    \begin{pmatrix}
        0 \\ \partial_{x_2}(\psi^2)
    \end{pmatrix}
    +
    \frac{1}{2}
    \Delta_j 
    \operatorname{div} ( \nabla^{\perp} \psi \otimes \nabla^{\perp} \psi )
\end{align}
and
\begin{align}\label{lower-est-2}
    \left\| 
    \Delta_j(-\Delta)^{-1}\mathbb{P}
    \begin{pmatrix}
        0 \\ \partial_{x_2}(\psi^2)
    \end{pmatrix}
    \right\|_{L^2(\mathbb{R}^2)}
    \geqslant c
\end{align}
for some positive constant $c$ independent of $j$.
Thus, by \eqref{lower-est-1} and \eqref{lower-est-2}, we have
\begin{align}
    &
    \eta^2\delta^2
    \sum_{-\frac{1}{2\delta^2} \leq j \leq {-2}}
    \n{
    \Delta_j
    (-\Delta)^{-1}
    \mathbb{P}
    \div
    \sp{ g \otimes g}}_{L^2(\mathbb{R}^2)}\\
    &\quad
    \geq{}
    c
    M^2\eta^2\delta^2
    \sum_{-\frac{1}{2\delta^2} \leq j \leq {-2}}
    \left\| 
    \Delta_j(-\Delta)^{-1}\mathbb{P}
    \begin{pmatrix}
        0 \\ \partial_{x_2}(\psi^2)
    \end{pmatrix}
    \right\|_{L^2(\mathbb{R}^2)}\\
    &\qquad
    -
    C
    \eta^2\delta^2
    \sum_{-\frac{1}{2\delta^2} \leq j \leq {-2}}
    \n{\Delta_j 
    (-\Delta)^{-1}
    \mathbb{P}
    \operatorname{div} ( \nabla^{\perp} \psi \otimes \nabla^{\perp} \psi )}_{L^2(\mathbb{R}^2)}\\
    &\quad 
    \geq 
    cM^2\eta^2
    -
    C\eta^2
    \n{\nabla^{\perp} \psi \otimes \nabla^{\perp} \psi}_{\dB_{2,\infty}^{-1}(\mathbb{R}^2)}\\
    &\quad 
    \geq 
    cM^2\eta^2
    -
    C\eta^2
    \n{\nabla^{\perp} \psi}_{L^2(\mathbb{R}^2)}^2,
\end{align}
which implies
\begin{align}
    \n{u_{\delta}^{(2;1)}(t_0+k_{\delta,T}T)}_{\dB_{2,1}^0(\mathbb{R}^2)}
    \geq 
    c_1M^2\eta^2
    -
    C_6\eta^2
\end{align}
for some positive constant $c_1$ and $C_6$.
Here, we have used the estimate 
\begin{align}
    \n{\nabla^{\perp} \psi \otimes \nabla^{\perp} \psi}_{\dB_{2,\infty}^{-1}(\mathbb{R}^2)} \leq C\n{\nabla^{\perp} \psi \otimes \nabla^{\perp} \psi}_{L^1(\mathbb{R}^2)} \leq C\n{\nabla^{\perp} \psi}_{L^2(\mathbb{R}^2)}^2.
\end{align}
Using \eqref{est:u_1_XY}, \eqref{est:u_2_2_XY}, \eqref{est:u_2_3_XY}, \eqref{est:u_3_XY}, and \eqref{est:rem}, 
we obtain
\begin{align}
    &\n{u_{\delta}[a](t_0+k_{\delta,T}T)}_{\dB_{2,1}^0(\mathbb{R}^2)}\\
    &\quad \geq{}
    \n{u_{\delta}^{(2;1)}(t_0+k_{\delta,T}T)}_{\dB_{2,1}^0(\mathbb{R}^2)}\\
    &\qquad-
    \n{u_{\delta}^{(1)}}_{\widetilde{L^{\infty}}(I_{\delta};\dB_{2,1}^{0}(\mathbb{R}^2))}
    -
    \sum_{k=2}^3
    \n{u_{\delta}^{(2;k)}}_{\widetilde{L^{\infty}}(I_{\delta};\dB_{2,1}^{0}(\mathbb{R}^2))}
    -
    \n{\widetilde{u}_{\delta}[a]}_{\widetilde{L^{\infty}}(I_{\delta};\dB_{2,1}^{0}(\mathbb{R}^2))}\\
    &\quad\geq{}
    c_1M^2\eta^2
    -C_6\eta^2\\
    &\qquad
    -3C_1M\eta\delta
    -C_2M^2\eta^3
    -C_2M^2\eta^2\delta^2
    -4C_0\varepsilon
    -4C_3M^8\eta^3\label{est:below}\\
    &\quad\geq{}
    c_1M^2\eta^2
    -C_6\eta^2\\
    &\qquad
    -3C_1M\eta^2
    -C_2M^2\eta^3
    -C_2M^2\eta^4
    -4C_0\varepsilon
    -4C_3M^8\eta^3\\
    &\quad=
    \sp{
    c_1M^2
    -C_6
    -3C_1M
    }
    \eta^2\\
    &\qquad
    -
    \sp{
    C_2M^2\eta
    +C_2M^2\eta^2
    +4C_3M^8\eta
    }\eta^2
    -
    4C_0\varepsilon.
\end{align}
We now choose 
$M = M_0 \geq 10$ 
and 
$0< \eta=\eta_0 \leq \min\{\eta_1,\eta_2\}$, 
so that 
\begin{gather}
    c_1M_0^2
    -C_6
    -3C_1M_0 \geq 3,\\
    C_2M_0^{2}\eta_0
    +C_2M_0^{2}\eta_0^2
    +4C_3M_0^{8}\eta_0
    \leq1,\qquad
    C_0M_0\eta_0 \leq 1,
\end{gather}
and let
\begin{align}
    \varepsilon_0
    :=
    \min
    \left\{
    \varepsilon_1,\varepsilon_2,\eta_0, \frac{\eta_0^2}{4C_0}, \frac{\eta_0^2}{2}
    \right\}.
\end{align}
Then, for any $0 < \delta \leq \varepsilon_0$, there holds
\begin{gather}
    \n{a}_{\dB_{2,1}^0(\mathbb{R}^2)}
    \leq 
    \varepsilon_0,\qquad
    \n{f_{\delta}}_{\widetilde{L^{\infty}}(\mathbb{R};\dB_{p,1}^{\frac{2}{p}-3}(\mathbb{R}^2))}
    \leq
    \delta,\\
    \n{u_{\delta}[a](t_0+k_{\delta,T}T)}_{\dB_{2,1}^0(\mathbb{R}^2)}
    \geq 2\varepsilon_0,
\end{gather}
and we complete the proof.
\end{proof}
\subsection{Unconditional uniqueness}
To complete the proof of Theorem \ref{thm:2}, we need the following unconditional uniqueness for the initial value problem \eqref{eq:NS_I}.
\begin{prop}\label{prop:unique}
    Let $I=[t_0,T_0) \subset \mathbb{R}$ and $a \in \dB_{2,1}^{0}(\mathbb{R}^2)$. 
    If two vector fields ${u_1,u_2} \in C(I;\dB_{2,1}^{0}(\mathbb{R}^2))$ are mild solutions to \eqref{eq:NS_I} with ${u_1}(t_0)={u_2}(t_0)=a$ and a same external force, then it holds ${u_1=u_2}$.
\end{prop}
\begin{rem}
For the unconditional uniqueness of the incompressible Navier--Stokes equations,
\cite{Fur-Lem-Ter-00} considered the three dimensional case and showed the uniqueness in the class $C(I;L^3(\mathbb{R}^3))$.
Their method is directly applicable to the general higher dimensional case $C(I;L^n(\mathbb{R}^n))$ with $n \geq 3$, whereas {its proof} fails in the two-dimensional case since the key embedding $L^{\frac{n}{2}}(\mathbb{R}^n) \hookrightarrow \dot{W}^{-1,n}(\mathbb{R}^n)$ does not valid when $n=2$. 
In Proposition \ref{prop:unique}, we find that the unconditional uniqueness holds in the slightly narrower class $C(I;\dB_{2,1}^0(\mathbb{R}^2))$ than $C(I;L^2(\mathbb{R}^2))$ by following the idea of \cite{Fur-Lem-Ter-00}.  
\end{rem}
To show Proposition \ref{prop:unique}, 
we first establish some bilinear estimates.
\begin{lemm}\label{lemm:4}
Let $I=[t_0,T_0) \subset \mathbb{R}$ be an interval.
Then,
there exists a positive constant $C$ such that 
\begin{align}
    &
    \sup_{t\in I}\n{
    \int_{t_0}^t
    e^{(t-s)\Delta}
    \mathbb{P}\div(u(s) \otimes v(s)) ds
    }_{\dB_{2,\infty}^0(\mathbb{R}^2)}\\
    &\quad\leq 
    C
    \sup_{t\in I}
    \n{u(t)}_{\dB_{2,1}^0(\mathbb{R}^2)}
    \sup_{t\in I}
    \n{v(t)}_{\dB_{2,\infty}^0(\mathbb{R}^2)}
\end{align}
for all $u \in C(I;\dB_{2,1}^0(\mathbb{R}^2))$ and $v \in C(I;\dB_{2,\infty}^0(\mathbb{R}^2))$.
\end{lemm}

\begin{proof}
As there holds 
\begin{align}
    &
    \n{\Delta_j 
    \int_{t_0}^t
    e^{(t-s)\Delta}
    \mathbb{P}\div(u(s) \otimes v(s)) ds}_{L^2(\mathbb{R}^2)}\\
    &\quad 
    \leq{} 
    C
    \int_{t_0}^t
    e^{-\frac{1}{4}2^{2j}(t-s)}2^j\n{\Delta_j(u(s) \otimes v(s))}_{L^2(\mathbb{R}^2)}
    ds\\
    &\quad 
    \leq{}
    C
    \int_{t_0}^t
    e^{-\frac{1}{4}2^{2j}(t-s)}ds
    2^j\sup_{t_0\leq s\leq t} \n{\Delta_j(u(s) \otimes v(s))}_{L^2(\mathbb{R}^2)}\\
    &\quad 
    ={}
    C
    2^{-j}\sup_{s \in I} \n{\Delta_j(u(s) \otimes v(s))}_{L^2(\mathbb{R}^2)},
\end{align}
we have
\begin{align}
    \sup_{t\in I}\n{\int_{t_0}^t
    e^{(t-s)\Delta}
    \mathbb{P}\div(u(s) \otimes v(s)) ds}_{\dB_{2,\infty}^0(\mathbb{R}^2)}
    \leq 
    C
    \sup_{t\in I}
    \n{u(t)\otimes v(t)}_{\dB_{2,\infty}^{-1}(\mathbb{R}^2)}.
\end{align}
Hence, it suffices to show 
\begin{align}
    \n{fg}_{\dB_{2,\infty}^{-1}(\mathbb{R}^2)}
    \leq
    C
    \n{f}_{\dB_{2,1}^0(\mathbb{R}^2)}
    \n{g}_{\dB_{2,\infty}^0(\mathbb{R}^2)}
\end{align}
for all $f \in \dB_{2,1}^0(\mathbb{R}^2)$ and $g \in \dB_{2,\infty}^0(\mathbb{R}^2)$.
To prove this, we use the Bony para-product decomposition:
\begin{align}
    fg=T_fg +R(f,g)+T_gf{.}
\end{align}
See appendix \ref{sec:a} for the definitions $T_fg$ and $R(f,g)$.
It follows from Lemma \ref{lemm:para} and the continuous embeddings $\dB_{2,\infty}^0(\mathbb{R}^2) \hookrightarrow \dB_{\infty,\infty}^{-1}(\mathbb{R}^2)$ and $\dB_{2,1}^0(\mathbb{R}^2) \hookrightarrow \dB_{2,\infty}^0(\mathbb{R}^2)$ that 
\begin{align}
    &
    \begin{aligned}
    \n{T_fg}_{\dB_{2,\infty}^{-1}(\mathbb{R}^2)}
    \leq{}&
    C
    \n{f}_{\dB_{\infty,\infty}^{-1}(\mathbb{R}^2)}
    \n{g}_{\dB_{2,\infty}^0(\mathbb{R}^2)}\\
    \leq{}&
    C
    \n{f}_{\dB_{2,\infty}^0(\mathbb{R}^2)}
    \n{g}_{\dB_{2,\infty}^0(\mathbb{R}^2)}\\
    \leq{}&
    C
    \n{f}_{\dB_{2,1}^0(\mathbb{R}^2)}
    \n{g}_{\dB_{2,\infty}^0(\mathbb{R}^2)}
    \end{aligned}
\end{align}
and similarly 
\begin{align}
    \n{T_gf}_{\dB_{2,\infty}^{-1}(\mathbb{R}^2)}
    \leq{}&
    C
    \n{g}_{\dB_{\infty,\infty}^{-1}(\mathbb{R}^2)}
    \n{f}_{\dB_{2,\infty}^0(\mathbb{R}^2)}\\
    \leq{}&
    C
    \n{g}_{\dB_{2,\infty}^0(\mathbb{R}^2)}
    \n{f}_{\dB_{2,\infty}^0(\mathbb{R}^2)}\\
    \leq{}&
    C
    \n{f}_{\dB_{2,1}^0(\mathbb{R}^2)}
    \n{g}_{\dB_{2,\infty}^0(\mathbb{R}^2)}.
\end{align}
By the continuous embedding $\dB_{1,\infty}^0(\mathbb{R}^2) \hookrightarrow \dB_{2,\infty}^{-1}(\mathbb{R}^2)$ and Lemma \ref{lemm:para}, we see that
\begin{align}
    \n{R(f,g)}_{\dB_{2,\infty}^{-1}(\mathbb{R}^2)}
    \leq
    C
    \n{R(f,g)}_{\dB_{1,\infty}^0(\mathbb{R}^2)}
    \leq
    C
    \n{f}_{\dB_{2,1}^0(\mathbb{R}^2)}
    \n{g}_{\dB_{2,\infty}^0(\mathbb{R}^2)}.
\end{align}
{C}ombining the above three estimates, we complete the proof.
\end{proof}

\begin{lemm}\label{lemm:5}
Let $I=[t_0,T_0) \subset \mathbb{R}$ be an interval.
Then, there exists a positive constant $C$ such that 
\begin{align}
    &
    \sup_{t \in I}
    \n{\int_{t_0}^t
    e^{(t-s)\Delta}
    \mathbb{P}\div(u(s) \otimes v(s)) ds}_{\dB_{2,\infty}^0(\mathbb{R}^2)}\\
    &\quad
    \leq{}
    C
    \sup_{t\in I}(t-t_0)^{\frac{1}{4}}\n{u(t)}_{\dB_{4,1}^0(\mathbb{R}^2)}
    \sup_{t\in I}\n{v(t)}_{\dB_{2,\infty}^0(\mathbb{R}^2)}
\end{align}
for all $u \in C((t_0,T_0);\dB_{4,1}^0(\mathbb{R}^2))$ and $v \in C(I;\dB_{2,\infty}^0(\mathbb{R}^2))$.
\end{lemm}
\begin{proof}
By the smoothing estimate for the kernel of $e^{(t-s)\Delta}\mathbb{P}\div$ that 
\begin{align}
    &\n{\int_{t_0}^t
    e^{(t-s)\Delta}
    \mathbb{P}\div(u(s) \otimes v(s)) ds}_{\dB_{2,\infty}^0(\mathbb{R}^2)}\\
    &\quad\leq{}
    C
    \int_{t_0}^t
    \n{e^{(t-s)\Delta}\mathbb{P} \div(u(s) \otimes v(s))}_{\dB_{2,\infty}^0(\mathbb{R}^2)} ds\\
    &\quad\leq{}
    C
    \int_{t_0}^t
    (t-s)^{-\frac{3}{4}}
    \n{u(s) \otimes v(s)}_{\dB_{2,\infty}^{-\frac{1}{2}}(\mathbb{R}^2)}
    ds\\
    &\quad\leq{}
    C
    \int_{t_0}^t
    (t-s)^{-\frac{3}{4}}(s-t_0)^{-\frac{1}{4}}
    ds
    \sup_{t_0 < s \leq t}(s-t_0)^{\frac{1}{4}}\n{u(s) \otimes v(s)}_{\dB_{2,\infty}^{-\frac{1}{2}}(\mathbb{R}^2)}
    \\
    &\quad\leq{}
    C
    \sup_{t_0 < s \leq t}(s-t_0)^{\frac{1}{4}}\n{u(s) \otimes v(s)}_{\dB_{2,\infty}^{-\frac{1}{2}}(\mathbb{R}^2)}.
\end{align}
Hence, it suffices to show 
\begin{align}
    \n{fg}_{\dB_{2,\infty}^{-\frac{1}{2}}(\mathbb{R}^2)}
    \leq
    C
    \n{f}_{\dB_{4,1}^0(\mathbb{R}^2)}
    \n{g}_{\dB_{2,\infty}^0(\mathbb{R}^2)}.
\end{align}
Similarly to the argument in the proof of Lemma \ref{lemm:4}, we use the para-product decomposition.
It follows from Lemma \ref{lemm:para} and the continuous embeddings $\dB_{4,1}^0(\mathbb{R}^2) \hookrightarrow \dB_{4,\infty}^0(\mathbb{R}^2) \hookrightarrow \dB_{\infty,\infty}^{-\frac{1}{2}}(\mathbb{R}^2)$ and $\dB_{2,\infty}^0(\mathbb{R}^2) \hookrightarrow \dB_{4,\infty}^{-\frac{1}{2}}(\mathbb{R}^2)$ that 
\begin{align}
    &
    \n{T_fg}_{\dB_{2,\infty}^{-\frac{1}{2}}(\mathbb{R}^2)}
    \leq
    C
    \n{f}_{\dB_{\infty,\infty}^{-\frac{1}{2}}(\mathbb{R}^2)}
    \n{g}_{\dB_{2,\infty}^0(\mathbb{R}^2)}
    \leq
    C
    \n{f}_{\dB_{4,1}^0(\mathbb{R}^2)}
    \n{g}_{\dB_{2,\infty}^0(\mathbb{R}^2)}\\
    &
    \n{T_gf}_{\dB_{2,\infty}^{-\frac{1}{2}}(\mathbb{R}^2)}
    \leq
    C
    \n{g}_{\dB_{4,\infty}^{-\frac{1}{2}}(\mathbb{R}^2)}
    \n{f}_{\dB_{4,\infty}^0(\mathbb{R}^2)}
    \leq
    C
    \n{g}_{\dB_{2,\infty}^0(\mathbb{R}^2)}
    \n{f}_{\dB_{4,1}^0(\mathbb{R}^2)}.
\end{align}
By the continuous embedding $\dB_{\frac{4}{3},\infty}^0(\mathbb{R}^2) \hookrightarrow \dB_{2,\infty}^{-\frac{1}{2}}(\mathbb{R}^2)$ and Lemma \ref{lemm:para}, we see that
\begin{align}
    \n{R(f,g)}_{\dB_{2,\infty}^{-\frac{1}{2}}(\mathbb{R}^2)}
    \leq
    C
    \n{R(f,g)}_{\dB_{\frac{4}{3},\infty}^0(\mathbb{R}^2)}
    \leq
    C
    \n{f}_{\dB_{4,1}^0(\mathbb{R}^2)}
    \n{g}_{\dB_{2,\infty}^0(\mathbb{R}^2)}.
\end{align}
combining the above three estimates, we complete the proof.
\end{proof}
Now, we present the proof of Proposition \ref{prop:unique}.
\begin{proof}[Proof of Proposition \ref{prop:unique}]
Let $u_1$ and $u_2$ be a solution to \eqref{eq:NS_I} with the same initial data $a\in {\dB_{2,1}^0}(\mathbb{R}^2)$.
Let $v_m(t):=u_m(t)-e^{(t-t_0)\Delta}a$ ($m=1,2$) and $w(t):= v_1(t) - v_2(t)$.
Then, it holds
\begin{align}
    w(t) = 
    &-
    \int_{t_0}^t
    e^{(t-s)\Delta}
    \mathbb{P}\div
    \sp{e^{(s-t_0)\Delta}a \otimes w(s)} ds\\
    &
    -
    \int_{t_0}^t
    e^{(t-s)\Delta}
    \mathbb{P}\div
    \sp{w(s) \otimes e^{(s-t_0)\Delta}a} ds\\
    &-
    \int_{t_0}^t
    e^{(t-s)\Delta}
    \mathbb{P}\div
    \sp{
    v_1(s) \otimes w(s)
    } ds\\
    &
    -
    \int_{t_0}^t
    e^{(t-s)\Delta}
    \mathbb{P}\div
    \sp{
    w(s) \otimes v_2(s)
    } ds.
\end{align}
It follows from Lemmas \ref{lemm:4} and \ref{lemm:5} that
{for any $T \in (t_0,T_0)$,} 
\begin{align}\label{pf:thm2:1}
\begin{split}
    \sup_{t_0<t<T}
    \n{w(t)}_{\dB_{2,\infty}^0(\mathbb{R}^2)}
    \leq{}&
    C_0
    \sup_{t_0<t<T}
    (t-t_0)^{\frac{1}{4}}
    \n{e^{(t-t_0)\Delta}a}_{\dB_{4,1}^0(\mathbb{R}^2)}
    \sup_{t_0<t<T}
    \n{w(t)}_{\dB_{2,\infty}^0(\mathbb{R}^2)}\\
    &
    +
    C_0
    \sum_{m=1}^2
    \sup_{t_0<t<T}
    \n{v_m(t)}_{\dB_{2,1}^0(\mathbb{R}^2)}
    \sup_{t_0<t<T}
    \n{w(t)}_{\dB_{2,\infty}^0(\mathbb{R}^2)}
\end{split}
\end{align}
with some positive constant $C_0$ independent of $T$.
As the density argument and $v_m(t_0)=0$ yields
\begin{align}
    \lim_{T\downarrow t_0}
    \sp{
    \sup_{t_0<t<T}
    (t-t_0)^{\frac{1}{4}}
    \n{e^{(t-t_0)\Delta}a}_{\dB_{4,1}^0(\mathbb{R}^2)}
    +
    \sum_{m=1}^2
    \sup_{t_0<t<T}
    \n{v_m(t)}_{\dB_{2,1}^0(\mathbb{R}^2)}
    }
    =0,
\end{align}
there exists a time $t_0 < T_1\leq {T_0}$ such that
\begin{align}\label{pf:thm2:2}
    \sup_{t_0<t<T_1}
    (t-t_0)^{\frac{1}{4}}
    \n{e^{(t-t_0)\Delta}a}_{\dB_{4,1}^0(\mathbb{R}^2)}
    +
    \sum_{m=1}^2
    \sup_{t_0<t<T_1}
    \n{v_m(t)}_{\dB_{2,1}^0(\mathbb{R}^2)}
    \leq 
    \frac{1}{4C_0}.
\end{align}
Then, we see by \eqref{pf:thm2:1} and \eqref{pf:thm2:2} that 
\begin{align}
\sup_{t_0<t<T_1}
\n{w(t)}_{\dB_{2,\infty}^0(\mathbb{R}^2)}
\leq 
\frac{1}{2}
\sup_{t_0<t<T_1}
\n{w(t)}_{\dB_{2,\infty}^0(\mathbb{R}^2)},
\end{align}
which implies $w(t)=0$ for all $t_0\leq t \leq T_1$.
If $T_1 <T_0$. then we repeat the same procedure many times to obtain $w(t)=0$ for all  $t \in I$.
Thus, we complete the proof.
\end{proof}
\subsection{Proof of Theorem \ref{thm:2}}
Now, we are in a position to present the proof of Theorem \ref{thm:2}.
\begin{proof}[Proof of Theorem \ref{thm:2}]
Let $1 \leq p \leq 2$ and let $\varepsilon_0$ be a positive constant determined in Proposition \ref{prop:non-per}.
Let $0<\delta \leq \varepsilon_0$ and $0<T\leq 2^{\frac{1}{\delta^2}}$.
Suppose by contradiction that the external force $f_{\delta}$ appearing in Proposition \ref{prop:non-per} generates a $T$-periodic mild solution $u_{{\rm per},\delta} \in C(\mathbb{R} ; \dB_{2,1}^0(\mathbb{R}^2))$ to \eqref{eq:NS_P} satisfying  
\begin{align}\label{est:u(t_0)}
    \n{u_{{\rm per},\delta}(t_0)}_{\dB_{2,1}^0(\mathbb{R}^2)} \leq \varepsilon_0
\end{align}
for some $t_0 \in \mathbb{R}$.
It is easy to see that $u_{{\rm per},\delta}$ is also a mild solution to \eqref{eq:NS_I} with the initial data $a=u_{{\rm per},\delta}(t_0)$.
On the other hand, by Proposition \ref{prop:non-per}, there exists a mild solution $u_{\delta}[u_{{\rm per},\delta}(t_0)] \in \widetilde{C}([t_0,t_0+k_{\delta,T}T];\dB_{2,1}^0(\mathbb{R}^2))$ to \eqref{eq:NS_I} with $a=u_{{\rm per},\delta}(t_0)$ satisfying
\begin{align}
    \n{u_{\delta}[u_{{\rm per},\delta}(t_0)](t_0+k_{\delta,T}T)}_{\dB_{2,1}^0(\mathbb{R}^2)} \geq 2\varepsilon_0,
\end{align}
which and \eqref{est:u(t_0)} imply 
\begin{align}
    u_{\delta}[u_{{\rm per},\delta}(t_0)](t_0)
    \neq
    u_{\delta}[u_{{\rm per},\delta}(t_0)](t_0+k_{\delta,T}T).
\end{align}
Using Proposition \ref{prop:unique}, we see by $u_{\delta}[u_{{\rm per},\delta}(t_0)](t_0) = u_{{\rm per},\delta}(t_0)$ that 
\begin{align}\label{uniq}
    u_{\delta}[u_{{\rm per},\delta}(t_0)](t)=u_{{\rm per},\delta}(t)
\end{align}
holds for all $t_0 \leq t \leq t_0 + k_{\delta,T}T$.
From \eqref{est:u(t_0)} and \eqref{uniq} with $t=t_0+k_{\delta,T}T$, we have
\begin{align}
    u_{{\rm per},\delta}(t_0)
    =
    u_{\delta}[u_{{\rm per},\delta}(t_0)](t_0)
    \neq 
    u_{\delta}[u_{{\rm per},\delta}(t_0)](t_0+k_{\delta,T}T)
    =
    u_{{\rm per},\delta}(t_0+k_{\delta,T}T)
\end{align}
which yields a contradiction to the periodicity of $u_{{\rm per},\delta}$. 
Thus, we complete the proof.
\end{proof}

\noindent
{\bf Data availability.} \\
Data sharing not applicable to this article as no datasets were generated or analyzed during the current study.

\noindent
{\bf Conflict of interest.} \\
The author has declared no conflicts of interest.

\noindent
{\bf Acknowledgments.} \\
The author was supported by Grant-in-Aid for Research Activity Start-up, Grant Number JP23K19011.

\appendix
\def\thesection{\Alph{section}}
\section{Remarks on the paradifferential calculus}\label{sec:a}
In this appendix, we consider the Bony decomposition of the product $fg$ for two functions $f$ and $g$:
\begin{align}
    fg = T_fg + R(f,g) + T_gf,
\end{align}
where we have set 
\begin{align}
    T_fg := 
    \sum_{k \in \mathbb{Z}} 
    \left( \sum_{\ell \leq k-3} \Delta_{\ell} f \right) 
    \Delta_k g,
    \qquad
    R(f,g) :=
    \sum_{k \in \mathbb{Z}}
    \sum_{|k-\ell|\leq 2}
    \Delta_kf 
    \Delta_{\ell}g.
\end{align}
We first recall the basic estimates for $T_fg$ and $R(f,g)$ in Besov and Chemin--Lerner spaces. 
\begin{lemm}\label{lemm:para}
Let $n \in \mathbb{N}$ and let $I \subset \mathbb{R}$ be an interval.
Then, the following statements hold:
\begin{enumerate}
    \item 
    Let $1 \leq p,p_1,p_2,r,r_1,r_2,\sigma,\sigma_1 \leq \infty$ and $s,s_1,s_2 \in \mathbb{R}$ satisfy
    \begin{align}
        \frac{1}{p} = \frac{1}{p_1} + \frac{1}{p_2},\quad
        \frac{1}{r} = \frac{1}{r_1} + \frac{1}{r_2},\quad
        {\frac{1}{\sigma} \leq \frac{1}{\sigma_1} + \frac{1}{\sigma_2},}\quad
        s=s_1+s_2, \quad 
        s_1 < 0.
    \end{align}
    Then, there exists a positive constant $K_1=K_1(\sigma_1,s_1,s_2)$ such that 
    \begin{align}\label{est:T-1}
        \n{T_fg}_{\dB_{p,\sigma}^s(\mathbb{R}^n)}
        \leq
        K_1
        \n{f}_{\dB_{p_1,\sigma_1}^{s_1}(\mathbb{R}^n)}
        \n{g}_{\dB_{p_2,\sigma}^{s_2}(\mathbb{R}^n)}
    \end{align}
    for all 
    $f \in \dB_{p_1,\sigma_1}^{s_1}(\mathbb{R}^n)$ 
    and 
    $g \in \dB_{p_2,\sigma}^{s_2}(\mathbb{R}^n)$,
    as well as 
    \begin{align}\label{est:T-2}
        \n{T_FG}_{\widetilde{L^r}(I;\dB_{p,\sigma}^s(\mathbb{R}^n))}
        \leq
        K_1
        \n{F}_{\widetilde{L^{r_1}}(I;\dB_{p_1,\sigma_1}^{s_1}(\mathbb{R}^n))}
        \n{G}_{\widetilde{L^{r_2}}(I;\dB_{p_2,\sigma}^{s_2}(\mathbb{R}^n))}
    \end{align}
    for all 
    $F \in \widetilde{L^{r_1}}(I;\dB_{p_1,\sigma_1}^{s_1}(\mathbb{R}^n))$ 
    and 
    $G \in \widetilde{L^{r_2}}(I;\dB_{p_2,\sigma}^{s_2}(\mathbb{R}^n))$.
    \item 
    Let $1 \leq p,p_1,p_2,r,r_1,r_2,\sigma,\sigma_1,\sigma_2 \leq \infty$ and $s,s_1,s_2 \in \mathbb{R}$ satisfy
    \begin{align}
        \frac{1}{p} = \frac{1}{p_1} + \frac{1}{p_2},\quad
        \frac{1}{r} = \frac{1}{r_1} + \frac{1}{r_2},\quad
        \frac{1}{\sigma} \leq \frac{1}{\sigma_1} + \frac{1}{\sigma_2},\quad 
        s=s_1+s_2>0.
    \end{align}
    Then, there exists a positive constant $K_2=K_2(s, s_1,s_2)$ such that 
    \begin{align}\label{est:R-1}
        \n{R(f,g)}_{\dB_{p,\sigma}^s(\mathbb{R}^n)}
        \leq
        K_2
        \n{f}_{\dB_{p_1,\sigma_1}^{s_1}(\mathbb{R}^n)}
        \n{g}_{\dB_{p_2,\sigma_2}^{s_2}(\mathbb{R}^n)}
    \end{align}
    for all 
    $f \in \dB_{p_1,\sigma_1}^{s_1}(\mathbb{R}^n)$ 
    and 
    $g \in \dB_{p_2,\sigma_2}^{s_2}(\mathbb{R}^n)$,
    as well 
    as 
    \begin{align}\label{est:R-2}
        \n{R(F,G)}_{\widetilde{L^r}(I;\dB_{p,\sigma}^s(\mathbb{R}^n))}
        \leq
        K_2
        \n{F}_{\widetilde{L^{r_1}}(I;\dB_{p_1,\sigma_1}^{s_1}(\mathbb{R}^n))}
        \n{G}_{\widetilde{L^{r_2}}(I;\dB_{p_2,\sigma_2}^{s_2}(\mathbb{R}^n))}
    \end{align}
    for all 
    $F \in \widetilde{L^{r_1}}(I;\dB_{p_1,\sigma_1}^{s_1}(\mathbb{R}^n))$ 
    and 
    $G \in \widetilde{L^{r_2}}(I;\dB_{p_2,\sigma_2}^{s_2}(\mathbb{R}^n))$.
\end{enumerate}
\end{lemm}
One may prove Lemma \ref{lemm:para} along the arguments in \cite{Bah-Che-Dan-11}*{Theorems 2.47 and 2.52}.
It follows from the proof of estimates \eqref{est:T-1}, \eqref{est:T-2} and \eqref{est:R-1}, \eqref{est:R-2} that there exists absolute positive constants $C_1$ and $C_2$ such that we may choose $K_1$ and $K_2$ as
\begin{align}\label{K_1-K_2}
    K_1(\sigma_1,s_1,s_2)
    ={}
    \frac{C_1}{s_1^{1 - \frac{1}{\sigma_1}}}
    2^{3|s_2|},\qquad
    K_2(s,s_1,s_2)
    ={}
    \frac{C_2^{|s_1|+|s_2|}}{s}.
\end{align}
While we see from \eqref{K_1-K_2} that $K_2 =O(s^{-1})$ as $s \downarrow 0$, 
we may relax the singularity $s^{-1}$ to $s^{-\frac{1}{\sigma}}$ if $\sigma_1$ and $\sigma_2$ satisfy a strict condition {in the following lemma, which is crucial in the proof of Lemma \ref{lemm:XYZ}}.
\begin{lemm}\label{lemm:para'}
Let $1 \leq p,p_1,p_2,r,r_1,r_2,\sigma,\sigma_1,\sigma_2 \leq \infty$ and $s,s_1,s_2 \in \mathbb{R}$ satisfy
\begin{align}
    \frac{1}{p} = \frac{1}{p_1} + \frac{1}{p_2},\quad
    \frac{1}{r} = \frac{1}{r_1} + \frac{1}{r_2},\quad
    1 \leq \frac{1}{\sigma_1} + \frac{1}{\sigma_2},\quad 
    s=s_1+s_2>0.
\end{align}
Then, there holds 
\begin{align}\label{est:R-2}
    \n{R(f,g)}_{\widetilde{L^r}(I;\dB_{p,\sigma}^s(\mathbb{R}^n))}
    \leq
    K_3
    \n{f}_{\widetilde{L^{r_1}}(I;\dB_{p_1,\sigma_1}^{s_1}(\mathbb{R}^n))}
    \n{g}_{\widetilde{L^{r_2}}(I;\dB_{p_2,\sigma_2}^{s_2}(\mathbb{R}^n))}
\end{align}
for all 
$f \in \widetilde{L^{r_1}}(I;\dB_{p_1,\sigma_1}^{s_1}(\mathbb{R}^n))$ 
and 
$g \in \widetilde{L^{r_2}}(I;\dB_{p_2,\sigma_2}^{s_2}(\mathbb{R}^n))$,
where the positive constant $K_3=K_3(\sigma,s, s_1,s_2)$ is given by 
\begin{align}
    K_3(\sigma,s,s_1,s_2)
    ={}
    \frac{C_3^{|s_1|+|s_2|}}{s^{\frac{1}{\sigma}}}
\end{align}
for some absolute positive constant $C_3$.
\end{lemm}
\begin{proof}
By $R(f,g)=R(g,f)$,
Applying $\Delta_j$ to $R(f,g)$, we see that 
\begin{align}
    \Delta_j
    R(f,g) 
    =
    \Delta_j
    \sum_{k \geq j-4}
    \sum_{|k-\ell|\leq 2}
    \Delta_kf 
    \Delta_{\ell}g,
\end{align}
which implies 
\begin{align}
    2^{sj}
    \n{\Delta_j
    R(f,g)}_{L^r(I;L^p(\mathbb{R}^n))}
    \leq{}&
    C2^{2|s_2|}
    \sum_{k \geq j-4}
    2^{s(j-k)}
    2^{s_1k}
    \n{\Delta_kf}_{L^{r_1}(I;L^{p_1}(\mathbb{R}^n))}\\
    &\quad\times
    \sum_{|k-\ell|\leq 2} 
    2^{s_2\ell}
    \n{\Delta_{\ell}g}_{L^{r_2}(I;L^{p_2}(\mathbb{R}^n))}.
\end{align}
Taking $\ell^{\sigma}(\mathbb{Z})$-norm of this and using the Hausdorff--Young inequality for the discrete convolution that 
\begin{align}
    &\n{
    R(f,g)
    }_{\widetilde{L^r}(I;\dB_{p,\sigma}^s(\mathbb{R}^n))}\\
    &\quad 
    \leq{}
    C2^{2|s_2|}
    \sp{
    \sum_{j \leq 4}
    2^{s\sigma j}
    }^{\frac{1}{\sigma}} \\
    &\qquad \times
    \n{
    \mp{
    2^{s_1k}
    \n{\Delta_kf}_{L^{r_1}(I;L^{p_1}(\mathbb{R}^n))}
    \sum_{|k-\ell|\leq 2}
    2^{s_2\ell}
    \n{\Delta_{\ell}g}_{L^{r_2}(I;L^{p_2}(\mathbb{R}^n))}}_{k\in \mathbb{Z}}
    }_{\ell^1(\mathbb{Z})}\\
    &\quad 
    \leq 
    C2^{2|s_2|}2^{4s}s^{-\frac{1}{\sigma}}
    \n{f}_{\widetilde{L^{r_1}}(I;\dB_{p_1,\sigma_1}^{s_1}(\mathbb{R}^n))}
    \n{g}_{\widetilde{L^{r_2}}(I;\dB_{p_2,\sigma_2}^{s_2}(\mathbb{R}^n))},
\end{align}
which completes the proof.
\end{proof}

\end{document}